\documentclass[11pt,reqno]{amsart}
\usepackage{amsmath,amsthm,amssymb,amsrefs} 
\usepackage{graphicx}
\usepackage{stmaryrd} 
 \usepackage{hyperref} 

\usepackage{verbatim}

\usepackage[usenames]{color}

\newtheorem{thm}{Theorem}[section]
\newcommand{\bt}{\begin{thm}}
\newcommand{\et}{\end{thm}}

\newtheorem{ex}[thm]{Example}

\newtheorem{cor}[thm]{Corollary} 
\newcommand{\bc}{\begin{cor}}
\newcommand{\ec}{\end{cor}}

\newtheorem{lem}[thm]{Lemma}
\newcommand{\bl}{\begin{lem}}
\newcommand{\el}{\end{lem}}

\newtheorem{prop}[thm]{Proposition}
\newcommand{\bp}{\begin{prop}}
\newcommand{\ep}{\end{prop}}

\newtheorem{defn}[thm]{Definition}

\newtheorem{rmrk}[thm]{Remark}

\newcommand{\GHto}{\stackrel { \textrm{GH}}{\longrightarrow} }
\newcommand{\Fto}{\stackrel {\mathcal{F}}{\longrightarrow} }

\newcommand{\be}{\begin{equation}}
\newcommand{\bee}{\begin{equation*}}

\newcommand{\ee}{\end{equation}}
\newcommand{\eee}{\end{equation*}}

\newcommand{\N}{\mathbb{N}}
\newcommand{\R}{\mathbb{R}}
\newcommand{\E}{\mathbb{E}} 

\newcommand{\diam}{\operatorname{Diam}}
\newcommand{\Sp}{\mathbb{S}}    

\newcommand{\vare}{\varepsilon}

\newcommand{\lip}{\operatorname{Lip}}
\newcommand{\mass}{{\mathbf M}}
\def\set{\textrm{set}}

\newcommand{\intcurr}{{\mathbf I}} 

\newcommand{\area}{\operatorname{Area}}
\newcommand{\vol}{\operatorname{Vol}}
\newcommand{\dvol}{\operatorname{dvol}}
\newcommand{\intcur}{{\mathbf I}}

\newcommand{\VFto}{\stackrel {\mathcal{VF}}{\longrightarrow} }

\def\Gr{\mathcal{G}_n(r_0,\gamma,D, \alpha)}
\newcommand{\Depth}{\operatorname{Depth}}

\begin{document}

\title[Volume Above Distance Below with Boundary]{Intrinsic Flat Stability of Manifolds with Boundary where Volume Converges and Distance is Bounded Below}

\author{Brian Allen}
\address{University of Hartford}
\email{brianallenmath@gmail.com}

\author{Raquel Perales}
\address{CONACyT Research Fellow-Universidad Nacional Aut\'onoma de M\'exico}
\email{raquel.peralesaguilar@gmail.com}

\begin{abstract}
Given a compact, connected, and oriented manifold with boundary $M$ and a sequence of continuous Riemannian metrics defined on it, $g_j$,
we prove volume preserving intrinsic flat convergence of the sequence to the smooth Riemannian metric $g_0$ provided $g_j$ always measures vectors strictly larger than or equal to $g_0$, the diameter of $g_j$ is uniformly bounded, the volume of $g_j$ converges to the volume of $g_0$, and $L^{\frac{m-1}{2}}$ convergence of the metrics restricted to the boundary. Many examples are reviewed which justify and explain the intuition behind these hypotheses. These examples also show that uniform, Lipschitz, and Gromov-Hausdorff convergence are not appropriate in this setting. Our results provide a new rigorous method of proving some special cases of the intrinsic flat stability of the positive mass theorem.
\end{abstract}

\maketitle

\section{Introduction}

We provide natural geometric hypotheses on Riemannian metrics, which can be obtained through standard geometric analysis techniques, that imply volume preserving intrinsic flat convergence to a particular Riemannian manifold with boundary. In the paper Volume Above Distance Below (VADB) of the authors and Sormani \cite{Allen-Perales-Sormani}, we prove the volume preserving intrinsic flat stability of the rigidity result which states that if $g_1 \ge g_2$ and $\vol(g_1)=\vol(g_2)$ for two Riemannian manifolds then $g_1=g_2$. Specifically it is shown that if $g_j \ge g_0$, $\diam(g_j) \le D$, and $\vol(g_j) \rightarrow \vol(g_0)$ for a sequence of Riemannian manifolds on a smooth compact oriented manifold without boundary $M$ then $(M,g_j)$ converges to $(M,g_0)$ in the volume preserving intrinsic flat sense. Our goal in this paper is to extend the work of the authors and Sormani \cite{Allen-Perales-Sormani}
to the setting of manifolds with boundary. 

\begin{thm}\label{vol-thm-boundary} 
Suppose we have a fixed  compact, oriented, connected, and smooth Riemannian manifold with non empty boundary, $M_0=(M^m,g_0)$,
a sequence of continuous Riemannian manifolds  $M_j=(M,g_j)$ so that
\be
\left( 1 - \frac{1}{j} \right) g_0(v,v) \le g_j(v,v),\quad \forall p \in M, v \in T_pM,
\ee
 a uniform upper bound on diameter
\be
\diam(M_j) \le D,
\ee
volume convergence
\be
\vol(M_j) \to \vol(M_0),
\ee
and
\be
 \| g_j|_{\partial M}-g_0|_{\partial M }\|_{L^{\frac{m-1}{2}}(\partial M,h)}\rightarrow 0.
\ee
Then $M_j$  converges to $M_0$ in the volume preserving intrinsic flat sense:
\be
M_j \VFto M_0.
\ee
\end{thm}

In fact, if we are willing to impose the stronger assumption of $L^{\frac{m}{2}}$ convergence of $g_j$ to $g_0$ we can get away with a much weaker assumption on the boundary.

\begin{thm}\label{vol-thm-boundary-improved} 
Suppose we have a fixed  compact, oriented, connected, and smooth Riemannian manifold with non empty boundary, $M_0=(M^m,g_0)$,
a sequence of continuous Riemannian manifolds  $M_j=(M,g_j)$ 
so that
\be
\left( 1 - \frac{1}{j} \right) g_0(v,v) \le g_j(v,v),\quad \forall p \in M, v \in T_pM, 
\ee
 a uniform upper bound on diameter
\be
\diam(M_j) \le D,
\ee
$L^{\frac{m}{2}}$ convergence
\be
\int_{M}|g_j-g_0|_{g_0}^{\frac{m}{2}} dV_{g_0} \rightarrow 0,
\ee
and
\be
\vol(\partial M_j) \le A.
\ee
Then $M_j$  converges to $M_0$ in the volume preserving intrinsic flat sense:
\be
M_j \VFto M_0.
\ee
\end{thm}

Intrinsic flat convergence was defined by Sormani-Wenger in \cite{SW-JDG} building upon the work of Ambrosio-Kirchheim \cite{AK}. A sequence of compact oriented manifolds $M_j$ converges in the intrinsic flat sense to $M_0$ if all the manifolds can be embedded by distance preserving maps $\varphi_j: M_j \to Z$ into a common complete metric space $Z$ so that the submanifolds $\varphi_j(M_j)$  converge in the flat sense as currents in $Z$  to  $\varphi(M_0)$ \cite{SW-JDG}.  The sequence is said to converge in the volume preserving intrinsic flat sense if additionally  $\vol(M_j) \to \vol(M_0)$.   It should be noted that for a sequence of Riemannian manifolds with bounded diameter, volume, and area of the boundaries a subsequence will converge in the intrinsic flat sense to an integral current space by Wenger's compactness theorem \cite{Wenger-compactness}. The importance of our theorem is that we can guarantee that the limiting space is a Riemannian manifold with a prescribed metric tensor.  

 We remark that Sormani-Wenger showed that if for a sequence both the Gromov-Hausdorff and the intrinsic flat limit exist then the intrinsic flat limit  is contained in the Gromov-Hausdorff limit, but the existence of one limit does not necessarily imply the existence of the other  \cite{SW-JDG}.

When working with a compact Riemannian manifold with boundary as opposed to one without boundary the global metric geometry, even for points away from the boundary, can change quite drastically. This is due to the fact that distance minimizing geodesics for points $p,q$ away from the boundary can still pass through the boundary. Given the importance of distance preserving maps in the definition of intrinsic flat distance this added complication can make arguments in the no boundary case break down when applied to the boundary case. In fact, to prove convergence of sequences of manifolds with boundary with sectional or Ricci curvature bounds 
it has been necessary to include assumptions on the mean curvature or second fundamental form of the boundary. 
Anderson, Katsuda, Kurylev, Lassas, and Taylor \cite{AndersonEtAl} and Knox \cite{Knox}  show convergence  in the H\"{o}lder space and Sobolev space
imposing uniform two sided  bounds on the Ricci or sectional curvature of $M_j$ and $\partial M_j$, a uniform upper diameter bound, a two sided uniform bound on the mean curvature of $\partial M_j$ and several other bounds.   Kodani \cite{Kodani} and  Wong \cite{Wong} show Gromov-Hausdorff (GH) convergence.    In particular, Wong imposes Ricci uniform lower bounds, two sided bounds on the second fundamental form and an upper diameter bound.  

In \cite{Perales2015}, the second named author shows intrinsic flat convergence as a direct application of  Wenger's compactness theorem. She imposes nonnegative Ricci curvature, an upper diameter bound, a bound on the volume of $\partial M_j$ and one sided bound on the mean curvature of $\partial M_j$.   Under this conditions,  a uniform bound on $\vol(M_j)$  is found and thus one gets a $\mathcal H^n$-countably rectifiable integral current  space  as limit space  \cite{Perales2015}.   See the survey by the second named author \cite{Perales2013survey} for an overview of results on convergence of manifolds and metric spaces with boundary.   In  \cite{Perales2014}, the second named author defines a class of manifolds with boundary with Ricci lower bounds which is precompact with respect to both Gromov-Hausdorff and intrinsic flat distance and the  limits coincide. But the requirements on the boundaries have a metric flavor.  This is based in the joint work of the second named author with Sormani \cite{Perales-Sormani}. In this paper we note that in Theorem \ref{convBdry} we give a restatement of Theorem \ref{vol-thm-boundary} where we replace the $L^{\frac{m-1}{2}}(\partial M_j,h)$ assumption with the assumption $\vol(\partial M_j) \le A$ and the condition that for any two interior points in $M_0$, the geodesic realizing their distance remains in the interior. 

The proof of Theorem \ref{vol-thm-boundary} applies some important results used to prove the main theorem in \cite{Allen-Perales-Sormani} with some added complications coming from the metric structure on manifolds with boundary. 
The first ingredient is to obtain pointwise almost everywhere subconvergence of the distance functions $d_j$ to $d_0$. In \cite{Allen-Perales-Sormani} this was done by constructing regions of $M$ which are foliated by distance realizing geodesics of $M_0$. Since distance realizing geodesics on manifolds with boundary do not foliate regions as in the no boundary case this construction does not easily extend to the boundary case. So instead we adapt a construction of Bray \cite{Bray-Penrose} to construct doublings of each manifold along its boundary while attaching necks to get manifolds with no boundary. This construction is interesting in its own right and should have further applications in the future. Here it is important that we retain some uniform control on the metric structures for these $\delta$-doubled metrics as well as the ability to compare the metric tensors over the entire sequence, which is something we show. Thus, after establishing important results for these $\delta$-doubled manifolds we are able to proceed as in \cite{Allen-Perales-Sormani}, that is we apply a result by the first named author and Sormani that implies almost everywhere subconvergence of the distance 
functions for the $\delta$-doubled metrics with no boundary.

The second main ingredient uses the pointwise almost everywhere subconvergence of distance functions to exhibit 
 metric spaces $(Z,d_Z)$  in which $M_j$ and $M_0$ can be isometrically embedded so one can calculate their flat distance and conclude intrinsic flat subconvergence of $M_j$. In this case we see that both the construction of $(Z,d_Z)$ and the intrinsic flat subconvergence result 
performed in \cite{Allen-Perales-Sormani} hold for manifolds with boundary.

The hypotheses of Theorem \ref{vol-thm-boundary} in the case of a sequence which does not necessarily include a boundary have been well developed by the first named author and Sormani in \cite{Allen-Sormani, Allen-Sormani-2} which culminated in the work of the authors and Sormani \cite{Allen-Perales-Sormani}. Many important examples have been given which justify and explain the intuition behind the hypotheses. For example, a few examples were given in \cite{Allen-Sormani, Allen-Sormani-2} which show that if all of the other hypotheses are satisfied except for $g_j \ge g_0$ then the limit in general may not even be a Riemannian manifold. Also, if all hypotheses are satisfied except for volume convergence then the limiting metric may be a Finsler manifold. It is strongly encouraged to read these examples first, which are reviewed in Section \ref{Examples}, to gain intuition before reading the remainder of the paper.

An important rigidity result involving scalar curvature is the positive mass theorem of Schoen and Yau \cite{Schoen-Yau-positive-mass}
and Witten~\cite{Witten-positive-mass}. The intrinsic flat stability of the positive mass theorem was conjectured by Lee and Sormani \cite{LeeSormani1} and has been shown in the rotationally symmetric case \cite{LeeSormani1}, in the graph case \cite{HLS}, and in various other cases.  Since this conjecture necessarily involves Riemannian manifolds with boundary we expect this paper to be useful for furthering our understanding of this conjecture. Here we prove stability of graphs in Euclidean space for the class  studied in  \cite{HLS}. 
For this we use the variant of Theorem \ref{vol-thm-boundary} that allows the interior of $M_0$ to be convex, allows the sequence to be continuous Riemannian manifolds, and requires a bound on the volumes of $\partial M_j$, Theorem \ref{convBdry}. The original proof in \cite{HLS} has some steps at the end that are difficult to follow and may require deep new theorems about integral current spaces to make them completely rigorous.  Here we avoid such complications by studying the sequence of manifolds themselves rather than their limit spaces.

The paper is organized as follows.    In Section~\ref{background} we provide a brief introduction to integral currents, integral current spaces and intrinsic flat distance. We also review some key results from the paper of the authors and Sormani \cite{Allen-Perales-Sormani}.   In Section~\ref{Examples} we review several examples of sequences of manifolds with boundary and discuss their convergence under intrinsic flat distance. These examples help to explain the intuition behind the hypotheses of Theorem \ref{vol-thm-boundary}.   In Section~\ref{VADBrev} we revise VADB to ensure that  under the assumption of almost everywhere subconvergence, $d_j \to d_0$,  the construction of the metric spaces $(Z,d_Z)$ and the intrinsic flat subconvergence result obtained by embedding the manifolds in these spaces hold when having non empty boundary.   In Section~\ref{Doubling} we introduce the $\delta$-doubling construction, establish many important estimates for this construction, and use it to prove that under the conditions of Theorem \ref{vol-thm-boundary}  we can find almost everywhere subconvergence of the distance functions $d_j  \to d_0$.  This is done by doubling $M$ to get  a manifold with no boundary and constructing Riemannian metrics that satisfy 
the conditions of the main theorem of the first named author and Sormani \cite{Allen-Sormani-2}.   In Section~\ref{Proofs} we prove Theorem \ref{vol-thm-boundary} and Theorem \ref{vol-thm-boundary-improved} by applying  Theorem \ref{PrelimMainThm} .   In Section ~\ref{HLSapplication} we explore an application of the work of this paper to the intrinsic flat stability of the positive mass theorem for Riemannian manifolds which are given as graphs of functions in Euclidean space.
\smallskip

{\bf{Acknowledgements}}:  The authors would like to thank Christina Sormani for her constant support and encouragement.


\section{Background}\label{background}

In this section we review the minimum background necessary to understand integral currents and intrinsic flat distance. 
References are given for the interested reader who would like to dig deeper into the technical details. 

We also state a few results from VADB. In particular, the ones that let us take $F_j$ to be the identity function in Theorem \ref{vol-thm-boundary} and the main theorem of this paper \cite{Allen-Perales-Sormani}, Theorem \ref{vol-thm}.  In Section \ref{VADBrev} we will review VADB  in a deeper way.

\subsection{Notation}

Here the term  isometric embedding and distance preserving map between metric spaces is used interchangeably. 
We index  sequences with subscripts $j$. For subsequences  we use indexes $j(k)$ where we see 
$j:  \mathbb N  \to  \mathbb N$ as an increasing function.   
For a Riemannian manifold $(M,g_j)$ we use the Riemannian metric or its index to denote which tensor we use, for example,
$\diam(M_j)=\diam_j(M)$,  $\vol(M_j)=\vol_j(M)$, $B_{g_j}(p,r)=B_{d_j}(p,r)$, where the latter 
denotes an open ball with radius $r$ and center $p$ with respect to the length distance $d_{g_j}=d_j$ induced by $g_j$. 
Moreover, when calculating intrinsic flat distance we are formally considering the 
integral current space $(M_j,d_j, [[M_j]])$ as defined in Subsection \ref{ssec-IFdist}.

\subsection{Review of Sormani-Wenger Intrinsic Flat Distance}\label{ssec-IFdist}

In this section we review the definition of integral currents in metric spaces and some of their properties 
 provided by Ambrosio-Kirchheim  \cite{AK}.   Based on this work Sormani-Wenger
defined integral current spaces and intrinsic flat distance.  We refer the reader to Sormani and Wenger 
\cite{SW-JDG,SW-CVPDE} and Sormani \cite{Sormani-ArzAsc} for related results.

\smallskip 
Given a complete metric space $(Z,d)$,  we denote by $\lip(Z)$ the set of  real valued Lipschitz functions on $Z$ and by
$\lip_b(Z)$ the bounded ones.  An $m$-dimensional current $T$ on $Z$ 
is a multilinear map $T: \lip_b(Z) \times [\lip(Z)]^{m}  \to \R$  that satisfies certain properties, see \cite[Definition 3.1]{AK}.  
From the definition of $T$ we know there exists a finite Borel measure on $Z$,  $||T||$,  called the mass measure of $T$.
Then the mass of $T$ is defined as $\mass(T)=||T||(Z)$.  The  boundary of $T$,  $\partial T: \lip_b(Z) \times [\lip(Z)]^{m-1}  \to \R$  
is the linear functional given by
\bee
\partial T(f, \pi) = T(1, (f, \pi)),
\eee
and for any Lipschitz function $\varphi: Z \to Y$  the push forward of $T$,   ${\varphi}_{\sharp} T : \lip_b(Y) \times [\lip(Y)]^{m}  \to \R$
is the $m$-dimensional current given by 
\bee
{\varphi}_{\sharp} T (f, \pi) 
= T( f\circ \varphi, \pi \circ \varphi ).
\eee
Furthermore,  the following inequality holds
\begin{equation}\label{eq-pushMeasure}
||  \varphi_\sharp T||  \leq \lip(\varphi)^m  \varphi_\sharp ||T||. 
\end{equation} 
 
The main examples of $m$-dimensional currents on $Z$ are currents that can be written as 
\be\label{block}
\varphi_{\sharp} [[\theta]]
(f, \pi)=  \int_{A} \theta(x) f(\varphi(x))  \det( D_x(\pi \circ\varphi) d\mathcal L^m(x),
\ee
where  $\varphi: A \subset \R^m \to  Z$  is a Lipschitz function, $A$  is a Borel set and, $\theta \in L^1(A, \mathbb R)$. 
An $m$-dimensional integral current in $Z$ is an $m$-dimensional current that can be written as a countable sum of terms as in (\ref{block}), 
$T=  \sum_{i=1}^\infty \varphi_{i\sharp} [[\theta_i]]$, with $\theta_i \in L^1(A_i, \mathbb R)$ integer constant functions, 
such that $\partial T$ is a current.  The class that contains all $m$-dimensional integral currents of $Z$ is denoted by $\intcurr_m(Z)$. 
For  $T\in I_m(Z)$,  Ambrosio-Kirchheim proved that the subset 
\begin{align}
\set(T)= \{ z \in Z \, | \, \liminf_{r \downarrow 0} \frac{\|T\|(B_r(z))}{ r^m }> 0 \}
\end{align}   
is $\mathcal H^m$-countably recitifiable. That is, $\set(T)$ can be covered by images of Lipschitz maps from $\R^m$  to $Z$
up to a set of zero  $\mathcal H^m$-measure.

The flat distance between two integral currents $T_1, T_2  \in  \intcurr_{m} (Z)$ is defined as
\begin{align*}
d_{F}^Z( T_1, T_2)=    \inf \{  \mass(U)+ \mass(V)\,  |&  \, \, U \in \intcurr_{m}(Z), \, V   \in  \intcurr_{m+1} (Z),\\
 & \, \, T_2 -T_1=  U + \partial V  \}.
\end{align*}

\smallskip

An $m$-dimensional integral current space $(X, d, T)$ consists of a metric space $(X, d)$ and an $m$-dimensional integral current defined on the completion of $X$, $T\in I_m(\bar{X})$, such that $\set(T)=X$.  There is also the notion of $m$-dimensional integral current space denoted as ${\bf{0} }$, here $T=0$ and $\set(T)=\emptyset$.

\begin{ex}
For an  $m$-dimensional compact oriented $C^0$ Riemannian manifold $(M^m,g)$,
the triple $(M,d_g, [[M]])$ given as follows is an $m$-dimensional integral current space.
Here
$d_g$ is the length metric induced by $g$,
 \be \label{djdefn}
d_g(p,q) = \inf\{L_g(\gamma):\, \gamma(0)=p, \, \gamma(1)=q\}
\ee
where $\gamma$ is any piecewise smooth curve joining $p$ to $q$
and
\be\label{Ljdefn}
L_g(\gamma)=\int_0^1 g(\gamma'(t),\gamma'(t))^{1/2}\, dt.
\ee
Then $[[M]] :  \lip_b(M) \times [\lip(M)]^m \to \R$ is given by 
\begin{align}\label{eq-canonicalT}
[[M]] = &  \sum_{i,k} {\psi_i}_\sharp [[1_{A_{ik}}]] 
\end{align}
where we have chosen a $C^1$ locally finite atlas $\{(U_i, \psi_i)\}_{i \in \mathbb N}$ of $M$ consisting of positively oriented Lipschitz charts, 
$\psi_i :  U_i  \subset \R^n   \to M$  and $A_{ik}$  are precompact Borel sets such that $\psi_i(A_{ik})$ have disjoint images for all $i$ and $k$ and, cover $M$ $\mathcal H^m$-almost everywhere. In this case, $||[[M]]||= \dvol_g$. 
\end{ex}

We say that an integral current space $(X,d,T)$ is precompact if $X$ is precompact with respect to $d$.
Given two $m$-dimensional integral current spaces, $(X_1, d_1, T_1)$ and $(X_2, d_2, T_2)$, a current preserving isometry between them is a 
 metric isometry $\varphi: X_1 \to X_2$  such that $\varphi_\sharp T_1=T_2$.   The definition of intrinsic flat distance is as follows.

\begin{defn}[Sormani-Wenger \cite{SW-JDG}]\label{IF-defn} 
Given two $m$-dimensional precompact integral current spaces 
$(X_1, d_1, T_1)$ and $(X_2, d_2,T_2)$, the intrinsic flat distance  $d_{\mathcal{F}}\left( (X_1, d_1, T_1), (X_2, d_2, T_2)\right)$ 
between them is defined as
\begin{align*}
\inf  \{d_F^Z(\varphi_{1\sharp}T_1, \varphi_{2\sharp}T_2)|  \,
&  \textrm{$(Z,d_Z)$ complete} ,\, \textrm{$\varphi_j: X_j \to Z $  isom embeddings}\}.
\end{align*}
\end{defn}

The function $d_{\mathcal{F}}$ is a distance up to current preserving isometries 
and the mass functional, $\mass$, is lower semicontinuous with respect to this distance. 
Thus, we say that a sequence $(X_j,d_j,T_j)$ of $m$-dimensional integral current spaces 
converges in $\mathcal{VF}$ sense to $(X,d,T)$ if the sequence converges with respect to the intrinsic flat distance
to $(X,d,T)$ and the masses $\mass(T_j)$ converge to $\mass(T)$.  Wenger proved in \cite{Wenger-compactness} 
that the class of $m$-dimensional precompact integral current spaces with 
a uniform diameter bound, a uniform mass of their currents and their boundaries is compact with 
respect to $d_{\mathcal{F}}$.

\subsection{$L^p$ convergence of Riemannian metrics}\label{ssec-Lp}

We recall the following definition. For more details see \cite{Clarke}.

For a compact manifold $M$ with Riemannian metrics $g_j$ and $g$, the $L^p$ norm, $p>1$,
of $g_j$ with respect to $g$ is given by 
\be
\|g_j\|_{L^p_{g}}=  \left( \int_M  |g_j|_g^p  \dvol_g\right)^{1/p},
\ee
where 
\be 
|g_j|_{g}(p) = \left( \displaystyle\sum_{i=1}^m \lambda_i(p)^4\right)^{1/2}
\ee
and $\lambda_1^2,...,\lambda_m^2$
are the eigenvalues of $g_j$ with respect to an orthonormal basis of $g$, i.e. for each $p \in M$  let  $v_i \in T_pM$, $i=1,...,m$ such that $g(v_i,v_i)=1$ and
\be
g_j(v_i,v_i)= \lambda^2_i.
\ee
We say that a sequence of Riemannian metrics $g_j$  converges to $g_\infty$
in $L^p_{g}$ norm if 
\be 
\|g_j - g_\infty\|_{L^p_g}  \to 0 \quad \text{ as } \quad j \to \infty.
\ee 

In \cite{Allen-Sormani, Allen-Sormani-2}, the first named author and Sormani showed that $L^p$ convergence of sequences of Riemannian metrics is weaker than intrinsic flat convergence (See example \ref{Cinched-Torus} below). The intuition behind this observation is that when a sequence of Riemannian metrics converge in the $L^p$ sense they are allowed to be quite different on sets of measure zero. Along the sequence it is possible that distance realizing curves could live in these measure zero sets which can cause shortcuts between points. These shortcuts will not effect the $L^p$ limit but will drastically effect the intrinsic flat limit which more intimately takes into account the metric structure of the sequence. In example \ref{Cinched-Torus} for instance, the $L^p$ limit is a cylinder but the intrinsic flat limit is a cinched metric space due to the shortcut which is forming along the center of the cylinder.


\subsection{A Few Words About VADB}

Here we just state two lemmas that let us take $F_j$ to be the identity function in Theorem \ref{vol-thm-boundary} and remind the reader of the volume above distance below theorem of the authors and Sormani \cite{Allen-Perales-Sormani} for manifolds with no boundary, Theorem \ref{vol-thm}.

Theorem \ref{vol-thm} follows from Theorem \ref{PointwiseConvergenceAE} below 
and the construction of suitable metric spaces $(Z,d_{Z})$ in which $M_j$ and $M_0$ can 
be isometrically embedded and hence, one can get an estimate  of their intrinsic flat distance.  
In Section  \ref{RevisitingVADB}  we provide more details of this construction and the intrinsic flat distance estimate 
so we can apply them to sequences of manifolds with boundary.   In Section \ref{Doubling}  
we will double the manifolds with boundary to get manifolds with no boundary and thus will be able to apply Theorem \ref{PointwiseConvergenceAE}. 
\smallskip

The lemma below allows  us to pass from the condition 
$F_j:   (M,  d_{j}) \to (M,  d_{0})$ distance decreasing with $C^1$ inverse to 
an inequality at the level of the Riemannian tensors.  

\begin{lem}[Lemma 2.2 in VADB]\label{DistToMetric}
Let $M_j=(M,g_j)$ and $M_0=(M,g_0)$ be continuous Riemannian manifolds and  $F: M_j \rightarrow M_0$ be a 
biLipschitz map with a $C^1$ inverse then 
\begin{equation}
g_0 (v, v)  \leq g_j (dF^{-1}(v), dF^{-1}(v)) \qquad \forall v\in TM_0
\end{equation}
if and only if 
\be
d_0(F(p),F(q))\leq d_j(p,q)\qquad \forall p,q\in M_j.
\ee 
\end{lem}

Now we see that the inequality between the Riemannian metrics  implies that the identity map 
is biLipschitz. 

\begin{lem}[Lemma 2.3 in VADB]\label{MetricToDist}
Let $M_j=(M,g_j)$ and $M_0=(M,g_0)$ be compact Riemannian manifolds with continuous metric tensors such that 
\begin{equation}
g_0 (v, v)  \leq g_j (v,v) \qquad \forall v\in TM.
\end{equation}
Then, the identity map from $(M_j, d_j)$ to $(M_0, d_0)$  is
biLipschitz. 
\end{lem}

Combining  Lemma \ref{DistToMetric} with Lemma \ref{MetricToDist} 
we can replace the condition $F_j$  biLipschitz and distance non-increasing with a $C^1$ inverse 
with the identity function by replacing $g_j$ with its pushforward under $F_j$,   $F_j^*g_j$. 
\smallskip

\begin{thm}[Theorem 1.1  VADB]\label{vol-thm}   
Suppose we have a fixed  compact oriented smooth Riemannian manifold, $M_0=(M^m,g_0)$,
without boundary,  a sequence of continuous Riemannian metrics $g_j$ on $M$ defining $M_j=(M, g_j)$ and
a sequence of biLipschitz and distance non-increasing maps 
\be
F_j: (M_j,d_j) \to (M_0,d_0)
\ee
with a  $C^1$ inverse and a uniform upper bound on diameter
\be
\diam(M_j) \le D
\ee
and volume convergence
\be
\vol(M_j) \to \vol(M_0),
\ee
then $M_j$ converges to $M_0$ in the volume preserving intrinsic flat sense,
\be
M_j \VFto M_0.
\ee
\end{thm}

Theorem \ref{vol-thm} follows from Theorem \ref{PointwiseConvergenceAE} below 
and the construction of suitable metric  spaces $(Z,d_{Z})$ in which $M_j$ and $M_0$ can 
be isometrically embedded and hence get an estimate 
of their intrinsic flat distance.

\begin{thm}[Theorem 4.4 in Allen-Sormani]\label{PointwiseConvergenceAE}
Suppose we have a fixed closed and smooth Riemannian manifold, $M_0=(M,g_0)$,
and a sequence of continuous Riemannian metrics $g_j$ on $M$ defining $M_j=(M, g_j)$ such that
\be
g_j(v,v) \ge g_0(v,v) \qquad \forall v\in T_pM
\ee
and
\be
\vol(M_j) \to \vol(M_0)
\ee
then there exists a subsequence such that 
\be
\lim_{k\to \infty} d_{j(k)}(p,q) = d_0(p,q)  \qquad \dvol_{g_0} \times \dvol_{g_0} \textrm{ a.e. } (p,q). 
\ee
\end{thm}


\section{Examples of Sequences with Boundaries}\label{Examples}

In \cite{Allen-Sormani} and \cite{Allen-Sormani-2}, the first named author and Sormani presented a number of examples 
comparing and contrasting various notions of convergence for Riemannian manifolds.   The hypotheses in these works are very similar  
to the ones we impose in our results and help to illustrate their importance. 
The examples in \cite{Allen-Sormani} were warped products which were allowed to have boundary.  Hence, 
we review the warped product examples with boundary from this paper. 
The examples in \cite{Allen-Sormani-2} were conformal metrics which were not allowed to have boundary.    
We review some conformal examples from \cite{Allen-Sormani-2} but remove a small ball around a point in each manifold  to create manifolds with boundary.

Example \ref{Cinched-Torus}  and Example \ref{Cinched-Sphere}  show that the conditions $g_0 \leq g_j$  and 
$(1-1/2j)g_0 \leq g_j$  are necessary in  
Theorem~\ref{vol-thm-boundary} and Theorem~\ref{vol-thm-boundary-improved}. 
In  Example \ref{to-Finsler} and Example \ref{NoL^mConv} we see  that $\vol(M_j)  \to  \vol(M_0)$ cannot be weaken to $\vol(M_j)  \leq V$. In Example \ref{SingleRidge}
we see that if the $L^{\frac{m-1}{2}}$ convergence of the boundaries is removed but the interior of $M_0$ is convex then the sequence still converges to $M_0$. 
This can be seen as an application of our Theorem \ref{convBdry}.    Finally,   in Example \ref{VolControlDiamNotConvergent} we see that under the hypotheses of Theorem~\ref{vol-thm-boundary} the Gromov-Hausdorff limit does not necessarily agree with the intrinsic flat limit and in fact the Gromov-Hausdorff limit does not necessarily exist in general.  Hence, intrinsic flat convergence is the appropriate notion of convergence for our setting.

More examples of manifolds with boundary that converge in intrinsic flat sense can be found in  \cite{Perales2015} and  \cite{Perales2014}. 
In both works the hypotheses imposed include a lower Ricci curvature  bound.  In   \cite{Perales2015}  the goal was to show 
a uniform bound on $\vol(M_j)$ and then use Wenger's compactness theorem. In  \cite{Perales2014} the goal was to 
define a precompact  class of manifolds, with respect to both Gromov-Hausdorff and intrinsic flat distance, such that the limits coincide. The requirements on the boundaries there had a metric flavor.  

\subsection{Lower bound on the sequence of metrics removed}
In these examples  we have the opposite bound on the metrics, that is, $g_j \le g_0$, but the other hypotheses are satisfies: $\vol_j \to \vol_0$, $\diam_j \leq D$,  $\|g_j|_{\partial M}-g_0|_{\partial M}\|_{L^{\frac{m-1}{2}}(\partial M,h)}\rightarrow 0$. Here we see that $M_j$  converges to something other than $M_0$.  Thus we note the importance of the
$(1-1/2j)g_0 \leq g_j$ bounds of Theorem \ref{vol-thm-boundary}  .

In this first example we have a sequence of warped cylinders which have a shortcut around the center circle which causes the sequence to converge to a cinched metric space which is not a Riemannian manifold.

\begin{ex}[Example 3.4 in \cite{Allen-Sormani}]\label{Cinched-Torus}  
Let $M=[-\pi,\pi]\times \mathbb{S}^1$ and $g_0=  dr^2+ d\theta^2$. 
We define warped product Riemannian manifolds 
\begin{align}
M_j= ([-\pi,\pi]\times \mathbb{S}^1, dr^2+f_j(r)^2 d\theta^2)
\end{align}
with warping functions that are given as follows. 
Let $h:[-1,1]   \to  [h_0, 1]$ be a smooth even function such that $h(-1)=1$ with $h'(-1)=0$, decreasing to $h(0)=h_0\in (0,1]$ and then
increasing back up to $h(1)=1$, $h'(1)=0$. Define  
$f_j(r):[-\pi,\pi]\to [1,2]$ as
 \be
 f_j=
 \begin{cases}
 1 & r\in[-\pi,- 1/j]
 \\  h(jr) & r\in[- 1/j, 1/j]
 \\ 1 &r\in [1/j, \pi].
 \end{cases}
\ee
Then we have 
\begin{align}
g_j \le g_0, \quad \vol(M_j) \to \vol(M_0),  \quad \diam(M_j) \le \diam(M_0). 
\end{align}
But
\be
M_j \GHto M_\infty \textrm{ and } M_j \Fto M_\infty,
\ee
where $M_\infty= (M,  dr^2+f_\infty(r)^2 d\theta^2)$ with $f_\infty(0)=h_0$ and otherwise $f_\infty(r)=1$. 
\end{ex}

This second example is a sequence of conformal metric tensors on a sphere minus a ball centered at the south pole. Then the limit space is shrunk near the equator so that one obtains a cinched space as the intrinsic flat limit.

\begin{ex}[Example 3.1 in  \cite{Allen-Sormani-2}]\label{Cinched-Sphere}
Let $g_0$ be the standard round metric on the sphere ${\mathbb S}^m$.   Let $g_j=f_j^2 g_0$ be
metrics conformal to $g_0$ with smooth conformal factors, $f_j$,
that are radially defined from the north pole with a cinch at the equator as follows:  
 \be
 f_j(r)=
 \begin{cases}
 1 & r\in[0,\pi/2- 1/j]
 \\  h(j(r-\pi/2)) & r\in[\pi/2- 1/j, \pi/2+ 1/j]
 \\ 1 &r\in [\pi/2+ 1/j, \pi]
 \end{cases}
\ee
where $h:[-1,1]\rightarrow \R$ is an even function 
decreasing to $h(0)=h_0\in (0,1)$ and then
increasing back up to $h(1)=1$.   Then $\Sp^m_j =  (\Sp^m, g_j)$ satisfies 
\begin{align}
g_j \le g_0, \quad \vol(\Sp_j^m) \to \vol(\Sp_0^m),  \quad \diam(\Sp_j^m) \le \diam(\Sp_0^m).
\end{align}
But  $\Sp^m_j \Fto \Sp^m_{\infty}$, where  $\Sp^m_{\infty}=(\Sp_{\infty}, g_\infty)$ and 
$g_\infty=f_\infty^2g_0$ with $ f_{\infty}(r)=h_0$ for  $r=\pi/2$ and  $ f_{\infty}(r)=1$ otherwise. 
Thus,  $\Sp^m_{\infty}$  is not isometric to $\Sp_0^m$ since 
the distance, $d_\infty$, between pairs of points in the same hemisphere near the equator is achieved by geodesics which run to the
equator, then around inside the cinched equator, and then out again.  
\end{ex}

\subsection{Volume Convergence Removed}

In the next two examples  we have the lower bound on distance $g_0 \le g_j$,  $\diam_j \leq D$,  $\|g_j|_{\partial M}-g_0|_{\partial M}\|_{L^{\frac{m-1}{2}}(\partial M,h)}\rightarrow 0$,  but no volume convergence.  Hence, in the first example we obtain a limit space that is not even Riemannian but instead is a Finsler manifold with a symmetric norm that is not an inner product which is not even locally isometric to $M_0$ anywhere. In the second example we get a manifold with a bubble attached as the limit space. 

\begin{ex}[Example 3.12 in \cite{Allen-Sormani}]\label{to-Finsler}
Let $M=[-\pi,\pi]\times \mathbb{S}^1$  with warped metrics $g_0=dr^2 + d\theta^2$ and 
 $g_j=dr^2 + f_j(r)^2 d\theta^2$, where smooth $f_j: [-\pi, \pi]  \to [1,5]$ are defined as follows. 
Let $h$ be an even function such that $h(-1)=5$  
decreasing down to $h(0)=1$ and then increasing back up to $h(1)=5$.  
Consider the dense set $S$ in $[-\pi,\pi]$ given by 
\begin{eqnarray}\label{denseS}
 S&=&\left\{s_{i,j}=-\pi + \tfrac{2\pi i}{2^j}\,: \,  i=0,1,2,...,2^j-1,2^j ,\,\, j\in \mathbb{N}\right\}
 \end{eqnarray}
 and let $\delta_j=2^{-2j}$ for $j \in \mathbb N$. Then define functions $f_j$  that are cinched on $S$ as follows
  \be
 f_j(r)=
 \begin{cases}
  h((r-s_{i,j})/\delta_j ) & r\in [s_{i,j}-\delta_j, s_{i,j} +\delta_j] \textrm{ for } i =1,...,2^j-1
 \\ 5 & \textrm{ elsewhere. }
 \end{cases}
\ee
Then we have 
\be
g_0  \le g_j,  \quad  \diam_j \leq D,  \quad  \|g_j|_{\partial M}-g_0|_{\partial M}\|_{L^{\frac{m-1}{2}}(\partial M,h)}\rightarrow 0. 
\ee
But  $\vol(M_j) \to 20 \pi^2 \not = \vol(M_0)$ given that $f_j \to 5$ in $L^p$ sense for $p \geq 1$. 
This sequence converges in intrinsic flat sense to a Finsler space $M_\infty=(M, d_\infty)$
where $d_\infty$ is the R-stretched Euclidean taxi metric given by 
\be
d_{\infty}((s_1,\theta_1),(s_2,\theta_2))= \min \left\{  \sqrt{s^2 + 5^2 \theta^2} ,
 s\left( \tfrac{\sqrt{24}}{5} \right)+ \theta \right\}
\ee
with $s=|s_2-s_1|$ and $\theta=d_{{\mathbb S}^1}(\theta_1, \theta_2)$. 
This happens due to $f_j(r)$ converging pointwise  to 1 on the countable dense set  $S$ and pointwise to $5$ elsewhere, 
creating paths in $M_{\infty}$ that are shorter than the paths in $M_j$.
\end{ex}

\begin{rmrk}
Note that we modified the original definition of $S$ in the example above to include $-\pi$ and $\pi$, 
that is we added $i=0, 2^j$, so that the $L^{\frac{m-1}{2}}$ convergence of the boundary would be true. 
The proof of the properties satisfied by the sequence $M_j$ given in \cite{Allen-Sormani} also hold in this case.
\end{rmrk}

The next example consists of a sequence of conformal tori  that converge to a metric space which is isometric to a flat torus with a flat disk attached so that the boundary of the disk is identified with a point in the flat torus.  By removing a small ball around a point where the conformal factor equals $1$ we see that $\vol(M_j)  \to \vol(M_0)$ is necessary in Theorem \ref{vol-thm-boundary}. 

\begin{ex}[Example 3.5 in \cite{Allen-Sormani-2}]\label{NoL^mConv}
Let  $(\mathbb T^m, g_0)$ be a torus and $h_j:[1,2] \rightarrow [1, \infty)$ be a smooth, decreasing function so that $h_j(1) = j$, $h_j'(1)=h_j'(2)=0$, and $h_j(2) = 1$ so that
\begin{align}
    \frac{1}{j^m}\int_1^2h_j(s)^m s^{m-1}ds \rightarrow 0.\label{ConstructionHyp}
\end{align}
Given a point $p \in \mathbb T^m$,  consider the sequence of functions $f_j:  \mathbb T^m  \to [1,\infty)$ which are radially defined from $p$ by
\begin{equation}
f_j(r)=
\begin{cases}
j &\text{ if } r \in [0,1/j]
\\h_j(jr) &\text{ if } r \in [1/j,2/j]
\\ 1 & \text{ if } r \in (2/j,\sqrt{m}\pi].
\end{cases}
\end{equation}
Then the sequence $\mathbb T^m_j=(\mathbb T^m, f_j^2 g_0)$ satisfies 
\begin{align}
g_0 \leq g_j,  \quad \diam(\mathbb T^m_j) \leq 1+ \sqrt{m}\pi, \\
\vol( \mathbb T^m_j )  \to \vol_{g_0}(B(p,1))+\vol(\mathbb T_0^m). 
\end{align}
Furthermore,  it converges in intrinsic and flat sense to a torus with a bubble attached, 
\begin{align}
    \mathbb T^m_{\infty} = (\mathbb T^m \sqcup_{\{p \sim q\, |\, q\in  \partial \mathbb{D}^m\}} \mathbb{D}^m, d_\infty). 
\end{align}
\end{ex}

\subsection{$L^{\frac{m-1}{2}}$ convergence of the boundary removed}

Here we give a modification of Example 3.7 of \cite{Allen-Sormani} where we do not have $L^{\frac{m-1}{2}}$ convergence on the boundary but where the sequence still converges to the desired Riemannian manifold.   The convergence is obtained following the proof of the modified example but we could  apply Theorem \ref{convBdry} below or Theorem \ref{vol-thm-boundary-improved} to show intrinsic flat convergence. 

\begin{ex}\label{SingleRidge}
Let $M=[-\pi,\pi]\times \mathbb{S}^1$ and $g_0=dr^2+ d\theta^2$.
We define warped product Riemannian manifolds $M_j= ([-\pi,\pi]\times \mathbb{S}^1, dr^2+f_j(r)^2 d\theta^2)$, 
where $f_j:[0,\pi]\to [1,2]$ are the functions given by 
 \be
 f_j(r)=
 \begin{cases}
 h(jr) & r\in[0, 1/j]
 \\ 1 &r\in [1/j, \pi],
 \end{cases}
\ee
with   $h: [0,1]  \to [h_0,1] $ is a smooth function such that 
$h(0)=h_0\in (1,2]$ and then
decreasing down to $h(1)=1$, $h'(1)=0$. 
Here we see that
\begin{align}
g_0 \leq g_j,  \quad \diam(M_j) \leq D.
\end{align}
But 
\be
f_j \not\to 1 \textrm{ in } L^p(\mathbb{S}^1\times \{0\}) \quad p \ge 1
\ee
implies that 
\be
 \|g_j|_{\partial M}-g_0|_{\partial M}\|_{L^{\frac{m-1}{2}}(\partial M,h)}\not \to 0
 \ee
and yet $M_j \to M_0$ in both the GH and $\mathcal{F}$
sense.
\end{ex}

\subsection{$\mathcal{VF}$ Limit Different to GH Limit}

Once again by removing a small ball around a point where the conformal factor equals $1$ we see that this final example satisfies all the conditions of 
Theorem \ref{vol-thm-boundary} but the intrinsic flat limit is not equal to the Gromov-Hausdorff limit. 
Here the sequence of conformal tori converges in Gromov-Hausdorff sense to a torus with a segment attached.  We note that this example can be extended as in Examples 3.8 in \cite{Allen-Sormani-2} by adding more and more splines to give an example which satisfies the hypotheses of Theorem \ref{vol-thm-boundary} which does not have a Gromov-Hausdorff limit.

\begin{ex}[Example 3.7 in \cite{Allen-Sormani-2}]\label{VolControlDiamNotConvergent}
Let  $\mathbb T^m_0=(\mathbb T^m, g_0)$ be a torus and  $h_j:[1,2]\rightarrow  [1, \infty)$ be a smooth, decreasing function so that $h_j(1) = \frac{j}{1+\ln(j)}$  and $h_j(2) = 1$.
Given a point $p \in \mathbb T^m$ and $\eta > 1$,  consider the sequence of functions $f_j:  \mathbb T^m  \to [1,\infty)$ which are radially defined from $p$ by
\begin{equation}
f_j(r)=
\begin{cases}
\frac{j^{\eta}}{1+\ln(j)} &\text{ if } r \in [0,1/j^{\eta}]
\\ \frac{1}{r(1-\ln(r))} &\text{ if } r \in (1/j^{\eta},1/j]
\\ h_j(jr) &\text{ if } r \in (1/j,2/j]
\\ 1 & \text{ if } r \in (2/j,\sqrt{m}\pi].
\end{cases}
\end{equation}
Then the sequence $\mathbb T^m_j=(\mathbb T^m, f_j^2 g_0)$ satisfies 
\begin{align}
g_0 \leq g_j,  \quad \diam(\mathbb T^m_j) \leq  \ln(\eta) + \sqrt{m}\pi, \quad
\vol( \mathbb T^m_j )  \to \vol(\mathbb T_0^m).
\end{align}
Furthermore,  it converges in intrinsic flat sense to $\mathbb T^m_0$ and in Gromov-Hausdorff sense to 
$\mathbb T^m_0$ with a line of length $\ln(\eta)$ attached.
\end{ex}


\section{VADB Revisited}\label{VADBrev}

As we mentioned in the introduction, the proof of Theorem \ref{vol-thm-boundary} relies on two ingredients. 
The first one is to obtain almost everywhere subconvergence of the distance functions $d_j$ to $d_0$ and the second uses the almost everywhere subconvergence to construct complete metric  spaces $(Z,d_Z)$  in which $M_j$ and $M_0$ can be isometrically embedded so one can conclude intrinsic flat subconvergence of $M_j$.   In this section we take care of the second ingredient.  Thus the aim is to generalize Theorem 4.1 in VADB to 
 manifolds with boundary, Theorem  \ref{thm-IFconvC0B}.  The proof consists in checking that  all the steps carried on in VADB hold when having boundary.
 
 Explicitly,  to estimate the intrinsic flat distance between  $M_j$ and $M_0$ we construct a complete metric space $(Z,d_Z)$ where both manifolds isometrically embed via isometric embeddings
 $\varphi_j$ and $\varphi_0$. Then by the definition of intrinsic flat distance $d_\mathcal{F}(M_0, M_j)$  will be bounded above by  $d^Z_{F}( \varphi_{0\sharp} [[ M_0]],   \varphi_{j\sharp} [[M_j]])$.   Recall that  for any Riemannian manifold  $(M^m,g)$ we consider the $m$-dimensional integral current space $(M,d_g, [[M]])$  described in Section \ref{background} and when having a sequence $(M_j^m, g_j)$ we denote the length distance $d_{g_j}$ by $d_j$. 
Under the hypotheses of Theorem  \ref{thm-IFconvC0B},  we can select a subset $W_j  \subset M$ with volume as close as we want to the volume of $M_0$, and thus as close to the volume of $M_j$ for large $j$,  such that the $d_j$ and $d_0$ distances between pairs of points in $W_j$ are almost the same.  

Now $Z$ can be chosen to be the product of $M$ and an interval $I$ of length  $h_j$ that depends on the $\sup_{W_j \times W_j}|d_j-d_0|$ and the diameter of $M_0$
so that $h_j  \to 0$. To isometrically embed $M_j$ into $Z$ we will also glue a copy of $M$ to the top part of $M \times I$ by identifying points of $W_j$. By not identifying points of $M\setminus W_j$ to the top part of $Z$ we do not require control of distances on $M \setminus W_j$ but rather just a control on volume. This gives us flexibility to allow distances to become quite large on the set $M \setminus W_j$ while still being able to estimate the intrinsic flat distance. Then $d_Z$ is chosen so that in the bottom it coincides with  $d_0$ and in the top it coincides with $d_j$. Then the flat distance $d^Z_{F}( \varphi_{0\sharp} [[ M_0]],   \varphi_{j\sharp} [[M_j]])$ will be bounded above by the difference of the volumes of  $M_j$ and $W_j$ plus the height $h_j$  times a bound on the volumes of $M_j$ and  $\partial M_j$. Thus, the flat distances will converge to zero. 

\subsection{New Theorems}

We start by stating the main result we prove in this section. In the next subsections 
we revise the proof for manifolds with no boundary and give appropriate modifications in order to conclude the proof of the theorem. 

\begin{thm}[c.f. Theorem 4.1 in VADB]\label{thm-IFconvC0B}
Let $M$ be a compact oriented manifold.   Let $M_0=(M,g_0)$ be a smooth Riemannian manifold and $M_j = (M,g_j)$ be continuous Riemannian manifolds
such that
\be 
g_0(v,v) \le g_j(v,v), \quad \forall v \in T_pM,
\ee
\be
\diam(M_j) \le D,
\ee
\be\label{volsconv}
\vol(M_j) \to  \vol(M_0),
\ee
\be
\vol(\partial M_j) \leq A,
\ee
and
\be\label{convae2}
d_j(p,q)\to d_0(p,q)  \qquad \dvol_{g_0} \times \dvol_{g_0} \textrm{ a.e. } (p,q). 
\ee
Then
\begin{align}
M_j \Fto M_0. 
\end{align}
\end{thm}

Combining Allen-Sormani's almost everywhere convergence Theorem \ref{PointwiseConvergenceAE}   and the previous result we get the following.

\begin{thm}\label{convBdry}
Let $M$ be a compact oriented manifold.   Let $M_0=(M,g_0)$ be a smooth Riemannian manifold and $M_j = (M,g_j)$ be continuous Riemannian manifolds
such that
\be 
g_0(v,v) \le g_j(v,v), \quad \forall v \in T_pM,
\ee
\be
\diam(M_j) \le D,
\ee
\be\label{volsconv}
\vol(M_j) \to  \vol(M_0),
\ee
\be
\vol(\partial M_j) \leq A,
\ee
and  the interior of $M_0$ is convex, i.e. for all $p,q$ in the interior of $M$, 
and $\gamma:[0,1]  \to M$ $g_0$ minimizing geodesic joining $p$ to $q$ 
we have that $\gamma(I)$ remains in the interior of $M$.  Then
\begin{align}
M_j \Fto M_0. 
\end{align}
\end{thm}


\subsection{Revisiting VADB}\label{RevisitingVADB}

From the inequality in the Riemannian metrics,  almost everywhere pointed convergence
and a bound on the volumes  we can ensure the existence of sets with 
big volume and good distance estimates between points inside them. 

\begin{lem}[Lemma 4.6 and Lemma 4.9 in VADB]\label{unif-on-W}
Suppose we have a fixed closed smooth  Riemannian manifold, $M_0=(M,g_0)$,
and  a sequence of metric tensors $g_j$ on $M$ defining $M_j=(M, g_j)$ such that
\be 
g_0(v,v) \le g_j(v,v)  \quad \forall v \in T_pM,
\ee
\be
\vol(M_j) \leq V,
\ee
\be\label{ptwise-to-bulk-1}
d_j(p,q)\to d_0(p,q)  \qquad \dvol_{g_0} \times \dvol_{g_0} \textrm{ a.e. } (p,q). 
\ee
Then for any $\lambda  \in (0, \diam(M_0))$ and  $\kappa >1$, there exists 
 a set $W_{\lambda,\kappa} \subset M$ such that for all $p_1, p_2 \in W_{\lambda, \kappa}$
\be
|d_j(p_1,p_2)-d_0(p_1,p_2)| < 2 \lambda + 2\delta_{\lambda,\kappa,j},
\ee
where $\delta_{\lambda,\kappa,j} \to 0$ as $j \to \infty$,
and
\be
\vol_j(M \setminus W_{\lambda, \kappa}) \le  \frac{1}{\kappa}\vol_0(M)+|\vol_j(M)-\vol_0(M)|.
\ee
\end{lem}

\begin{proof}
The proof is as follows.  First fix $\lambda  \in (0, \diam(M_0))$ and $\kappa >1$. 

In {\textit{Lemma 4.8 in VADB}} by using the Bishop-Gromov volume comparison theorem it is shown that 
there exists $\vare= \vare(\lambda, \kappa) > 0$ small enough  so that $\kappa\vare \in (0,1/2)$ and  
\begin{equation}\label{volBound}
\min_{x\in M} \vol_0(B(x,\lambda)) \geq 2\kappa\vare \vol_0(M).
\end{equation}

In {\textit{Proposition 4.3 in VADB}} by applying Egoroff's theorem to the finite measure space $$(M_0 \times M_0,  \dvol_{g_0}  \times \dvol_{g_0})$$
and the functions  $d_j  \to d_0$ that satisfy (\ref{ptwise-to-bulk-1}) and $\vare$ as in the previous step, we 
get a measurable set $S_\vare\subset M_0\times M_0$ such that  
\be\label{unifSvare}
\sup\{|d_j(p,q)- d_0(p,q)|\,:\, (p,q)\in S_\vare\}=\delta_{\vare,j} \to 0
\ee
\be\label{volSvare}
\vol_{0\times 0} (S_\vare)> (1-\vare)\vol_{0\times 0}(M\times M),
\ee
and by enlarging the set we can assume that $(p,q)\in S_\vare$ if and only if $(q,p) \in S_\vare$. Hence we have obtained uniform control on the sequence of distance functions on $S_{\varepsilon} \subset M \times M$ but now we would like to use this to find a subset of $M$ where we have uniform control on the sequence of distance functions.

Thus in {\textit{Lemma 4.4 in VADB}} by (\ref{volSvare}) we get that 
for almost every $p  \in \pi_1(S_\vare)$ where $\pi_1:  M \times M \to M$ is projection onto the first factor,
the sets  $$S_{p,\vare}=\{q\in M\,:\, (p,q)\in S_\vare\}  =  \pi_1^{-1}(p)$$
are $\dvol_0$ measurable and satisfy 
\be\label{AverageAreaOfGoodSetInequality}
(1-\vare) \vol_0(M) <  (\vol_0(M))^{-1} \int_{p\in M} \vol_0(S_{p,\vare})\, \dvol_0.
\ee
So we see that on average the slices $S_{p,\varepsilon}$ have more volume than $(1-\varepsilon)\vol_0(M)$ which motivates the definition of our special subset of $M$ where we will be able to uniformly control distances.

In  {\textit{Lemma 4.5 in VADB}}  we define 
\be \label{Wkappavare}
W_{\kappa\vare}=\{p  \in \pi_1(S_\vare) \,:\, \vol_0(S_{p,\vare}) > (1- \kappa\vare) \vol_0(M)\}.
\ee
Using (\ref{Wkappavare}) and (\ref{AverageAreaOfGoodSetInequality}) it follows that 
\be\label{WkappavareVol}
\vol_0(W_{\kappa\vare}) > \frac{\kappa-1}{\kappa} \vol_0(M).
\ee
Hence $W_{\kappa \vare}\subset M$ is the set where we will now show uniform control of the sequence of distance functions.

In {\textit{Lemma 4.6 in VADB}}, 
from (\ref{WkappavareVol}), using that  $d_0 \leq d_j$  and so $\dvol_0  \leq  \dvol_j$ 
and $\vol(M_j) < \infty$ we get
\be\label{volskappa}
\vol_j(M \setminus W_{\kappa\vare}) \le \frac{1}{\kappa}\vol_0(M)+ |\vol_j(M)-\vol_0(M)|.
\ee
This shows that we have control of $\vol_j(M \setminus W_{\kappa\vare})$ as $\kappa \rightarrow \infty$ and $j \rightarrow \infty$.

Now in  {\textit{Lemma 4.7 in VADB}}  for $p_1, p_2  \in W_{\kappa\vare}$ distinct points, 
by  (\ref{Wkappavare}) and $\kappa \vare < 1/2$, we get 
$S_{p_1,\vare}  \cap  S_{p_2,\vare} \neq \emptyset$
and
\be\label{IntersecSpvare}
\vol_0(S_{p_1,\vare} \cap S_{p_2,\vare}) > (1-2\kappa\vare) \vol_0(M).
\ee

In  {\textit{Lemma 4.8 in VADB}} we show that 
for all $x \in M$ and $p_1,p_2\in W_{\kappa\vare}$, 
 from  (\ref{volBound}) and (\ref{IntersecSpvare}) it follows that 
 \be\label{IntersecNonEm}
B_{g_0}(x,\lambda) \cap S_{p_1,\vare} \cap S_{p_2,\vare} \neq \emptyset.
\ee

Finally,  in {\textit{Lemma 4.9 in VADB}} we argue that  for all $p_1, p_2 \in W_{\kappa\vare}$.
By (\ref{IntersecNonEm}) there exists 
$q \in B(p_1,\lambda) \cap S_{p_1,\vare} \cap S_{p_2,\vare}$. Using $d_0 \leq d_j$,  the triangle inequality, 
that $(p_i,q)$ satisfy  (\ref{unifSvare})  and that $d(p_1,q) < \lambda$,  we prove that for all $p_1, p_2 \in W_{\kappa\vare}$
\be\label{distWkappavare}
|d_j(p_1,p_2)-d_0(p_1,p_2)| < 2 \lambda + 2\delta_{\vare,j}.
\ee
Thus demonstrating the uniform control of distances for the set $W_{\kappa \varepsilon}$. Then (\ref{distWkappavare}) and (\ref{volskappa}) give the conclusion of the lemma. 
\end{proof}

\begin{rmrk}
Note that for the boundary case we only have to justify the first step of the previous proof. 
This is done as follows, since $(M_0, d_0)$ is compact there is a finite  cover of $M_0$
by balls of radius $\lambda/2$, $\{B_{d_0}(x_j, \lambda/2) \}$. Thus, for any $x \in M_0$ there exists a $j$ such 
that $d_0(x_j, x) \leq \lambda/2$. Hence, $B_{d_0} (x, \lambda) \supset B_{d_0} (x_j, \lambda/2)$.
Thus $\vol_0(B_{d_0} (x, \lambda))  \geq    \min_{j} \{ \vol_0( B_{d_0} (x_j, \lambda/2) \}$. 
Choose 
\begin{equation}
\vare=    \frac{  \min_{j} \{ \vol_0( B_{d_0} (x_j, \lambda/2) \}} {2\kappa \vol_0(M)}.
\end{equation} 
\end{rmrk}

Now we explain how to construct a metric space $(Z,d_Z)$ in which $M_0$ and $M_j$ isometrically embed. 
We will use this metric space to calculate the flat distance between the isometric images of $M_0$ and $M_j$. 
This will give us an upper bound on the intrinsic flat distance $d_{\mathcal F} (M_0, M_j)$, c.f. Definition \ref{IF-defn}.

\begin{defn}[VADB]\label{defn-Z}
Let $M$ be a compact manifold, $M_j=(M,g_j)$ and $M_0=(M,g_0)$ be continuous Riemannian manifolds, $F_j: M_j \rightarrow M_0$ a bijective map and $W_j \subset M_j$.   
\smallskip 

Define the set
\be
Z : =  M_0  \sqcup \left(   M \times [0,h_j] \right) \sqcup  M_j  \,\,|_\sim
\ee 
where  $x \sim (F_j^{-1}(x),0)$ for all $x \in M_0$ and  $x \sim (x,h_j)$ for all $x \in W_j$.

\smallskip
Define the function $d_Z: Z \times Z \to [0, \infty)$  by
\be
d_Z(z_1, z_2) = \inf \{L_Z(\gamma):\, \gamma(0)=z_1,\, \gamma(1)=z_2\}
\ee
where $\gamma$ is any piecewise smooth curve joining $z_1$ to $z_2$ and the length function 
$L_Z$ is defined as follows,  $L_Z|_{M_j} = L_{g_j}$,  $L_Z|_{M_0} = L_{g_0}$ and $L_Z|_{M \times (0,h_j]} = L_{g_j +  dh^2}$, see 
(\ref{Ljdefn}). 
\smallskip

Define functions $\varphi_0: M_0 \to Z$ and $\varphi_j: M_j \to Z$ by 
\begin{align}
\varphi_0(x) = & (F_j^{-1}(x), 0) \\
\varphi_j(x) =   &
\begin{cases}
x   & x \notin \overline{W}_j \\
(x, h_j)  &  \textrm{otherwise.} 
\end{cases}
\end{align}
\end{defn}

Now we give some estimates on the metric space $(Z,d_Z)$ which will allow us to show that $\varphi_0$ and $\varphi_j$
isometrically embedded  $M_0$ and $M_j$ into $Z$, respectively.

\begin{lem}[Lemma 3.3 and Lemma 3.5  in VADB] \label{cnstr0-Z} 
Let $M$ be a closed manifold,  $M_0=(M,g_0)$   and $M_j=(M, g_j)$ be continuous Riemannian manifolds. 
Let  $F_j: M_j \rightarrow M_0$ be a biLipschitz and distance non-increasing map with a $C^1$ inverse.   Then $(Z, d_Z)$ is a complete metric space and for all  $(x,h),(x',h') \in M\times[0,h_j] \subset Z$,
\be\label{dZEstimatetog_0}
d_Z((x,h),(x',h'))\ge \sqrt{d_0(F_j(x),F_j(x'))^2 +|h-h'|^2} 
\ee
and
\be \label{region-dist-dec-to-Z}
d_Z((x,h),(x',h')) \le \sqrt{d_j(x,x')^2 +|h-h'|^2}.
\ee 
Furthermore, if  $\diam(M_j) \le D$,  $W_j \subset M_j$ and  for all $x,y \in W_j$ 
\be\label{eq-distCond0}
d_j(x,y) \le d_0(F_j(x), F_j(y)) +2 \delta_j
\ee
for some $\delta_j >0$ and $h_j \ge \sqrt{2 \delta_j D + \delta_j^2}$, 
then $\varphi_0: M_0 \to Z$ and $\varphi_j: M_j \to Z$ are distance preserving. 
\end{lem}

We note that the way $h_j$ is chosen prevents having shorter paths  between pairs of points 
either in $M_0$ or $M_j$ seen as subsets of $Z$ than the ones in $M_0$ or $M_j$.

\begin{proof}
Here we review the proofs of Lemma 3.3 and Lemma 3.5 in VADB paying particular attention to the fact that $M_0$ and $M_j$ are continuous Riemannian metrics.

Let $C(t)=(\gamma(t),h(t))$, $t \in [0,1]$, be a curve connecting $(x,h),(x',h') \in M \times [0,h_j]\subset Z$ where it is enough to assume that $C(t) \subset M \times [0,h_j]$. Then by the definition of the length structure on Z we find
\begin{align}
\int_0^1 \sqrt{g_0(\gamma',\gamma')+ h'^2}dt\le L_Z(C) \le \int_0^1 \sqrt{g_j(dF^{-1}_j(\gamma'),dF^{-1}_j(\gamma'))+ h'^2}dt,\label{LengthInequality}
\end{align}
and hence by taking the infimum of lengths in \eqref{LengthInequality} we find the desired inequalities for the associated distance functions.

By \eqref{dZEstimatetog_0} it is clear that $\varphi_0(M_0)$ is distance preserving, as shown in Lemma 3.5 of VADB.

 Now by choosing $h_j \ge \sqrt{2\delta_jD+\delta_j^2}$, \eqref{eq-distCond0}, and the estimates on $d_Z$ we ensure that for points in $\overline{W}_j \times \{h_j\}$ it is never more efficient to take advantage of shortcuts in $M \times \{0\}\subset Z$  and hence for any curve $C$ connecting  points in $\overline{W}_j \times \{h_j\}$ we find
 \begin{align}
 L_{Z}(C) \ge d_j(p,q).\label{LengthDistanceInequality1}
 \end{align}

Then by the fact that points in $\varphi_j(M \setminus \overline{W}_j)$ are not glued to $M \times \{h_j\}$, and hence must enter $\overline{W}_j \times \{h_j\}$ before attempting to take advantage of shortcuts in $M \times\{0\}\subset Z$, we are able to conclude that for points in $\varphi_j(M \setminus \overline{W}_j)$ and curves $C$ connecting them
\begin{align}
 L_{Z}(C) \ge d_j(p,q).\label{LengthDistanceInequality2}
 \end{align}

Now for $p,q \in M_j$, by taking curves $C_i \subset \varphi_j(M_j)$ connecting $\varphi_j(p),\varphi_j(q)$ whose lengths converge to the distance $d_j(p,q)$ we can combine with \eqref{LengthDistanceInequality1} and \eqref{LengthDistanceInequality2} to find
\begin{align}
d_Z(   \varphi_j(p),\varphi_j(q) ) = d_j(p,q),
\end{align}
and hence we can conclude that $\varphi_j(M_j)$ is also distance preserving.
\end{proof}
  

Now we can calculate the flat distance between $\varphi_{j_\sharp}[[M_j]]$ and $\varphi_{0_\sharp}[[M_0]]$. 

\begin{thm}[c.f. Theorem 3.1 in  VADB]\label{est-SWIF}
Let $M$ be an oriented and compact manifold, $M_j=(M,g_j)$ and $M_0=(M,g_0)$ be continuous Riemannian manifolds with  $\diam(M_j) \le D$,  $\vol(M_j)\le V$, $\vol(\partial M_j) \leq A $
and $F_j: M_j \rightarrow M_0$ a biLipschitz  and distance non-increasing map with a $C^1$ inverse. 
Let $W_j \subset  M_j$ be a measurable set with 
\be\label{eq-volCond}
\vol( M_j \setminus W_j) \le V_j
\ee
and assume that there exists a $\delta_j > 0$ so that for all $ x,y \in W_j$,
 \be\label{eq-distCond}
d_j(x,y) \le d_0( F_j(x),  F_j(y)) +2 \delta_j
\ee
and that $h_j \ge \sqrt{2 \delta_j D + \delta_j^2}$.
Then
\be\label{Fest}
d^Z_{F}( \varphi_0(M_0), \varphi_j(M_j)) \le 2V_j + h_j V +  h_j A 
\ee
where  $Z$ is the space described in Definition \ref{defn-Z}.
\end{thm}

To not overburden the proof of Theorem \ref{est-SWIF} with notation we give a short proof below. 
In Section \ref{sec-app} we provide a detailed one.  The main difference is that the second one explains how the orientation comes into 
play in the definition of the various integral currents. 

\begin{proof}
Apply Lemma  \ref{cnstr0-Z}  to get a metric space $(Z,d_Z)$ and distance preserving maps $\varphi_0: M_0 \to Z$ and $\varphi_j : M_j  \to Z$.   
We will define integral currents $T \in  \intcur_{m+1}(Z) $ and $T' \in  \intcur_m(Z)$ such that 
\begin{align}
\varphi_{j\#}[[M_j]]- \varphi_{0\#}[[ M_0]]   & =  \partial T + T' \\
\mass(T)    & \leq  h_j V \\ 
 \mass(T')   & \leq  2V_j  +h_j A.
\end{align}
Then by the definition of flat convergence 
\begin{align}
d^Z_{F}( \varphi_j(M_j), \varphi_0( M_0)) \le  & \mass(T) + \mass(T')
\end{align}
and the mass estimates we will get (\ref{Fest}).
   
Since $\varphi_0$ and $\varphi_j$ are distance preserving maps,  
\begin{eqnarray}
\varphi_{0\#}[[ M_0]]  &  =  &  [[M_j \times \{0\}]] \\
\varphi_{j\#}[[ M_j]]  &  =  &  [[W_{j}  \times \{h_j \}  ]]  +  [[M_j \setminus W_{j}]]. 
\end{eqnarray}
Then define
\begin{align}
T=  & [[  \,M_j \times [0,h_j] \,]]   \in  \intcur_{m+1}(Z) \\
T'   =  & [[M_j \setminus W_{j}]]  -    [[   (M_j \setminus W_{j}  ) \times \{h_j \} ]]  - [ [\, \partial M_j \times [0,h_j] \,]]    \in  \intcur_{m}(Z) . 
\end{align}
They are integral currents since their boundaries are the following currents 
\begin{align}
\partial T = &  [[ M_j  \times \{h_j \}  ]]  -   [[M_j \times \{0\}]]  +  [ [\, \partial M_j \times [0,h_j] \,]] \\
\partial  T'  =  & \partial \left( [[M_j \setminus W_{j}]]  -  [[ (M_j \setminus W_{j}  ) \times \{h_j \} ]]  \right) -  \partial [ [\, \partial M_j \times [0,h_j] \,]]\\
 = & 0  -  \partial [ [\, \partial M_j \times [0,h_j] \,]].
\end{align}
Noticing that 
\begin{align}
T' = &  \varphi_{j\#}[[M_j]]   -     [[ M_j \times \{h_j \} ]]  - [ [\, \partial M_j \times [0,h_j] \,]]  
\end{align}
and combining the equations above, we conclude that 
\be
\varphi_{j\#}[[M_j]]- \varphi_{0\#}[[ M_0]] = \partial  T + T'. 
\ee
Now by the definition of $T$ and $T'$,
\begin{align}
\mass(T) \leq  & \vol(M_j \times [0,h_j]) \leq Vh_j\\
\mass(T')  \leq  & 2 \vol(M_j \setminus W_{j})    +   \vol( \partial M_j\times [0,h_j]) \\
    \leq & 2 V_j +   Ah_j.
\end{align}
\end{proof}


\subsection{Proofs  of  New Theorems}

We are ready to show intrinsic flat convergence to $M_0$ under the assumptions of Theorems  \ref{thm-IFconvC0B} - \ref{convBdry}.

\begin{proof}[Proof of Theorem \ref{thm-IFconvC0B}]
Since $\vol_j(M)  \to \vol_0(M)$, there is  $V>0$ that uniformly bounds  $\vol_j(M)$ from above. 
For any $\kappa >1$ and $\lambda \in  (0, \diam(M_0))$, since we have almost everywhere convergence of the distance functions, (\ref{convae2}),   we can apply Lemma \ref{unif-on-W}.   Thus there exists a set 
 $W_{\lambda, \kappa} \subset M$ such that for all $p_1, p_2 \in W_{\lambda, \kappa}$
 \be
|d_j(p_1,p_2)-d_0(p_1,p_2)| < 2 \lambda + 2\delta_{\lambda, \kappa, j}
\ee
where $\delta_{ \lambda, \kappa,  j}  \to 0$ as $j \to \infty$ and 
\be\label{vol1}
\vol_j( M \setminus W_{\lambda, \kappa,j}) \le  \frac{1}{\kappa}\vol_0( M)+|\vol_j( M)-\vol_0( M)|.
\ee

Now recalling the definition of intrinsic flat distance and applying  Lemma \ref{est-SWIF} we get,
\begin{align}
d_{\mathcal{F}}(M_0,M_j) \le  &  d^{Z_{\lambda, \kappa, j}}_{F}( \varphi_0 (M_0), \varphi_j(M_j)) \\
\le &  2\vol_j( M \setminus W_{\lambda, \kappa,j}) + h_j V  + h_j A  \\
\le & 2\left( \tfrac{1}{\kappa}\vol_0(M)+  |\vol_j(M)-\vol_0(M)| \right)  +   h_j  V     + h_j A 
\end{align}
where $h_j =   \sqrt{2(\lambda+\delta_{\lambda, \kappa, j})D  +   (\lambda+\delta_{ \lambda,  \kappa,   j})^2}$. 
Since the volumes converge \eqref{volsconv} and  $\delta_{ \lambda, \kappa,  j}  \to 0$ as $j \to \infty$,  we find
\begin{align}
\limsup_{j\rightarrow \infty} d_{\mathcal{F}}(M_0,M_j) \le  \tfrac{2}{\kappa} \vol_0(M)+   \sqrt {2 \lambda D  +   \lambda^2} (V + A) .
\end{align}
Since this is true for any $\kappa >1$ and $\lambda  \in  (0, \diam(M_0))$ we find that
\begin{align}
\limsup_{j\rightarrow \infty} d_{\mathcal{F}}(M_0,M_j) =0.
\end{align}
\end{proof}

Now we show convergence when the interior of $M_0$ is convex.  

\begin{proof}[Proof of Theorem \ref{convBdry}]
The theorem follows applying Theorem \ref{thm-IFconvC0B}. 
Thus, we only have to show that there exists a subsequence of $d_j$
that converges almost everywhere to $d_0$. 
We recall that in the proof of  Theorem \ref{PointwiseConvergenceAE} it was shown that 
for any $p \in M_0$ and all $q  \notin C(p)$, where $C(p)$ denotes the cut locus of $M_0$, there exists an open set $\mathcal{U}(p,q)   \subset M_0 \times M_0$
that contains $(p,q)$ such that 
\begin{align}
   \int_{\mathcal{U}(p,q)} |d_j(p',q') - d_0(p',q')|  \dvol_0\times \dvol_0  \rightarrow 0. 
   \label{IntegralToZero}
\end{align}
The set $\mathcal{U}(p,q)$ was found by choosing a tubular neighborhood around the minimizing geodesic from $p$ to $q$ 
in such a way the exponential map remained a diffeomorphism in the tubular neighborhood.   

Given that the interior of $M_0$ is convex, even if it has boundary, the same result holds 
for any $p \in M_0 \setminus \partial M_0$ and all $q  \notin C(p) \cup \partial M_0$.  
Now we proceed as in the final steps of the proof of Theorem \ref{PointwiseConvergenceAE}. We define
\begin{align}
    \mathcal{S}=    \bigcup_{p\in M}    \{p\}  \times (C(p) \cup \partial M_0). 
\end{align}
Since $\mathcal{S}$ has zero measure,
it is enough to show that  \eqref{convae2} holds for a subset of full measure in $M_0\times M_0 \setminus \mathcal{S}$.

Consider the map $\Psi: M_0\times M_0 \setminus \mathcal{S} \rightarrow \R$ given by 
$\Psi(p,q) = d_0(q,C(p))$.  Given that $M_0$ is compact and $\Psi$ is continuous, the sets
\begin{align}
K_i=\{(p,q) \in M\times M \setminus \mathcal{S}: \Psi(p,q)\ge 1/(1+i)\} \subset M\times M \setminus \mathcal{S}
\end{align}
are compact and satisfy  by definition $K_i \subset K_{i+1}$ and $\displaystyle \bigcup_{i=1}^{\infty} K_i = M_0 \times M _0 \setminus \mathcal{S}$. 
Note that $ \{\mathcal{U}{(p,q)}: (p,q) \in M_0\times M_0   \setminus \mathcal{S}\}$
is an open cover of $M_0\times M_0 \setminus \mathcal{S}$ and hence an open cover of $K_i$ for each $i \in \N$. Thus we can choose a finite subcover $\{\mathcal U_1,...,\mathcal U_{I_1}\}$ of $K_1$,  and then extend  it to a finite subcover of $K_2$, $\{\mathcal U_1,...,\mathcal U_{I_1},..., \mathcal U_{I_2}\}$, and continue in this way and get  a countable collection of elements  $\{\mathcal U_i\}_{i \in \N}$  so that
\begin{align}
M\times M\setminus \mathcal{S} \subset \bigcup_{i\in \N} \mathcal U_i.
\end{align}
By \eqref{IntegralToZero} we can choose a subsequence $d_j \rightarrow d_0$ 
that converges  almost everywhere on $\mathcal U_1$. Then we can choose a further subsequence $d_j \rightarrow d_0$ that converges almost everywhere on $\mathcal U_2$ and continue to build a nested sequence of subsequences that converges almost everywhere on each  $\mathcal U_i$. By extracting a diagonal subsequence we then obtain
\begin{align}
    d_{j}(p,q) \to d_0(p,q)    \quad  \dvol_{g_0}\times \dvol_{g_0} \text{a.e.} \,\,(p,q) \in  M\times M \setminus \mathcal S.
\end{align}

\end{proof}


\section{ $\delta$-Doubling  Metrics}\label{Doubling}

Let $M^m$ be a compact manifold with boundary,  $(M,g_j)$ a sequence of Riemannian manifolds and $(M,g_0)$ a background Riemannian manifold. 
Our goal is to show that under the conditions of Theorem \ref{vol-thm-boundary} with $F_j$ equal to the identity map 
we have up to a subsequence that
\begin{align}
    d_j(p,q) \rightarrow d_0(p,q)  \quad   \dvol_0 \times \dvol_0\,    \text{a.e.} \, (p,q) \in M \times M.
\end{align}
 Once we have this almost everywhere convergence we will be able to apply Theorem \ref{thm-IFconvC0B}
to conclude the proof of Theorem \ref{vol-thm-boundary}.   To accomplish our goal we will double $M$ by attaching two copies of $M$ to a neck of the form $\partial M \times [-\delta, \delta]$.
In this way we get manifolds $\hat M^\delta$   with no boundary. Then by defining Riemannian metrics $g_j^\delta$ that satisfy 
the conditions of  Allen-Sormani's Theorem \ref{PointwiseConvergenceAE},  we get  up to a subsequence that
\be
d^\delta_j(p,q) \rightarrow d^\delta_0(p,q)  \quad    \dvol_{g^\delta_0} \times \dvol_{g^\delta_0}\, \text{a.e.} \, (p,q) \in \hat M^\delta \times \hat M^\delta.
\ee
Then by the construction of $g^j_\delta$  we will be able to pass to the limit as $\delta$ goes to zero and get the conclusion.    

Before constructing the family of $\delta$-doubled Riemannian manifolds we will prove a lemma which allows us to realign the coordinates near the boundary of $M$ so that the normal vector with respect to $g_j$ and $g_0$ is the same vector in the tangent space.

\begin{lem}\label{RealignCoordsLem}
Let $M$ be a smooth, compact, oriented, and connected manifold with boundary. Assume $M_0=(M,g_0)$, $M_1=(M,g_1)$ are Riemannian manifolds so that
\begin{align}
g_0(v,v) < g_1(v,v), \quad \forall p \in M, v \in T_pM, v \not = 0,
\end{align}
then there exists a diffeomorphism
\begin{align}
F_1:M \rightarrow M,
\end{align}
so that 
\begin{align}
g_0(v,v) \le (F_1^*g_1)(v,v), \quad \forall p \in M, v \in T_pM,
\end{align}
and for $p \in \partial M$ we have that there exists a $\nu\in T_pM$ which is the inward pointing normal vector with respect to $g_0$ and $F_1^*g_1$.
\end{lem}

\begin{proof}
Let $\nu_0$ and $\nu_1$ be the inward pointing unit normal vectors to $\partial M = \Sigma$ with respect to $g_0$ and $g_1$, respectively. Then let $\Phi_i: \Sigma\times [0,\varepsilon) \rightarrow M$, $i=0,1$ be a one parameter family of smooth bijective maps which solve
\begin{align}
\frac{\partial \Phi_i}{\partial t} = \nu_i, i=0,1,
\end{align}
which we know has a solution for at least a short time $t \in [0,\varepsilon_i)$, $i=0,1$. Let $\varepsilon=\min(\varepsilon_0,\varepsilon_1)$ and define $h_i=g_i|_{\Sigma}$, $\Sigma_t^i=\Phi_i(\Sigma,t)$. If $h_i(z,t)$ is the evolution of the metric of $\Sigma_t^i$ defined on $\Sigma$ and $v \in T_z\Sigma$ then 
\begin{align}
h_i(z,t)(v,v)= (g_i)_{\Phi_i(z,t)}(d\Phi_i(v),d\Phi_i(v)).
\end{align}
Note that if $z\in \Sigma$ then $\Phi_0(z,0)=\Phi_1(z,0)$ and this point can be tracked to a point in $\Sigma_t^i$ through the map $\Phi_i(z,t)$. By assumption
\begin{align}
h_0(z,0)(v,v)<h_1(z,0)(v,v), \quad \forall v \in T_z\Sigma,
\end{align}
and hence there must exist an $\eta > 0$ so that 
\begin{align}
h_1(z,0)(v,v)-h_0(z,0)(v,v) > \eta, \quad \forall z \in \Sigma, v \in T_z\Sigma, |v|_{g_0}=1.
\end{align}
Since $\Phi_i$ are smooth maps and the unit tangent bundle over $\Sigma$ is compact there exists a $0<\varepsilon' < \varepsilon$ such that
\begin{align}
h_0(z,t)(v,v)<h_1(z,t)(v,v), \quad \forall z \in \Sigma, v \in T_z\Sigma, |v|_{g_0}=1, t \in [0,\varepsilon'],
\end{align}
which implies
\begin{align}
h_0(z,t)(v,v)<h_1(z,t)(v,v), \quad \forall z \in \Sigma, v \in T_z\Sigma, t \in [0,\varepsilon'].
\end{align}
Now we can define a portion of the map $F_1:M \rightarrow M$ using the coordinates on $\Sigma \times [0, \varepsilon']$ by
\begin{align}
F_1(\Phi_1(z,t)) &= \Phi_0(z,t),
\\ \Phi_1(z,t) &= F_1^{-1}(\Phi_0(z,t)),
\end{align}
so that
\begin{align}
(F_1^{-1})^*(\nu_0)= (F_1^{-1})^*\left(\frac{\partial \Phi_0}{\partial t}\right) = \frac{\partial}{\partial t}\left(F_1^{-1}(\Phi_0(z,t)) \right ) = \frac{\partial \Phi_1}{\partial t}  = \nu_1,
\end{align}
and
\begin{align}
(F_1)^*(g_1)_{\Phi_0(z,t)}&(d\Phi_0(v),d\Phi_0(v)) 
\\&= (g_1)_{\Phi_1(z,t)}(dF_1^{-1}(d\Phi_0(v)),dF_1^{-1}(d\Phi_0(v)))
\\&= (g_1)_{\Phi_1(z,t)}(d\Phi_1(v),d\Phi_1(v))
\\&= h_1(z,t)(v,v)
\\&\ge h_0(z,t)(v,v)= g_0(d\Phi_0(v),d\Phi_0(v)).
\end{align}

Now we want to extend the definition of $F_1$ to all of $M$ so that it remains a diffeomorphism and a distance non-increasing map. Let $0<t \le \varepsilon'$ and define  $V^i_t=\Phi_i(\Sigma \times [0,\min\{2t,\varepsilon'\}])$, $V_t = V_t^0 \cup V_t^1$, and $U_t=M \setminus V_t$ then we want to extend the definition of $F_1$ to all of $M$ so that for any $w \in T_pM, p\in M$
\begin{align}
(g_1)_{F_1^{-1}(p)}(dF_1^{-1}(w),dF_1^{-1}(w)) &\ge (g_0)_p(w,w),\label{Cond1}
\\F_1(\Phi_1(z,t)) &= \Phi_0(z,t),\quad (z,t) \in \Sigma \times [0,t],\label{Cond2}
\\ F_1(p)&=p, \quad p \in U_t.\label{Cond3}
\end{align}
Notice that we can clearly extend $F_1$ so that \eqref{Cond2} and \eqref{Cond1} are satisfied and we claim that there exists a $\varepsilon''$, $0 < \varepsilon''\le \varepsilon'$,   so that for $t \in (0,\varepsilon'')$ there exists a map $F_1$ which satisfies \eqref{Cond1} in addition to  \eqref{Cond2} and \eqref{Cond3}. For sake of contradiction assume that this is not the case, and let $F_1^t:M \rightarrow M$ be a map which extends $F_1$ to all of $M$ which satisfies \eqref{Cond2}, \eqref{Cond3}, and so that
\begin{align}
|\nabla^{g_1} d(F_1^t)^{-1}|_{g_1}+|d(F_1^t)^{-1}|_{g_1} \le C,\label{DiffBound}
\end{align}
but so that \eqref{Cond1} is not satisfied. We are justified in assuming \eqref{DiffBound} since we are requiring the map $F_1^{t}$ to differ less and less from the identity map as $t \rightarrow 0$. If we let $t_i$ be a sequence so that $t_i \in (0,\varepsilon')$ and $t_i \rightarrow 0$ then we know there is a subsequence of $(F_1^{t_i})^{-1}$ which converges in $C^1$ (with respect to $g_1$) to a map $F_{\infty}:M \rightarrow M$. By \eqref{Cond3} we find pointwise convergence of $(F_1^{t_i})^{-1}$ to the identity map and hence $F_{\infty}(p)=p$, $\forall p \in M$.

By the contradiction hypothesis we know there must exist a $p_i \in M$, $w_i \in T_{p_i}M$, $|w_i|_{g_0}=1$ so that 
\begin{align}
|d(F_1^{t_i})^{-1}(w_i))|^2_{(g_1)_{(F_1^{t_i})^{-1}(p_i)}} &< |w_i|^2_{(g_0)_{p_i}}\label{ContradictionEq1}
\end{align}

Now since $\{(p_i,w_i)\} \subset TM$, $|w_i|_{g_0}=1$,  and $M$ is compact we know there exists a subsequence which converges to $(p_{\infty},w_{\infty}) \in TM$, $w_{\infty} \not = 0$. Hence by \eqref{ContradictionEq1} we find
\begin{align}
|w_{\infty}|^2_{(g_1)_{p_{\infty}}} &= |d(F_{\infty})^{-1}(w_{\infty}))|^2_{(g_1)_{(F_{\infty})^{-1}(p_{\infty})}}
\le  |w_{\infty}|^2_{(g_0)_{p_{\infty}}},\label{ContradictionEq2}
\end{align}
which is a contradiction.
\end{proof}

Now we construct the family of $\delta$-doubling  Riemannian manifolds by modifying the doubling construction of Bray \cite{Bray-Penrose} in the proof of Theorem 9. 

\begin{thm}\label{FamilyOfMetrics}
Suppose we have a fixed smooth compact, oriented, and connected Riemannian manifold with boundary, $M_0=(M^m,g_0)$,
and  a sequence of smooth metric tensors $g_j$ on $M$ defining $M_j=(M, g_j)$.  Assume that
\begin{align}
g_0(v,v) < g_j(v,v), \quad \forall p \in M, v \in T_pM, v \not = 0.
\end{align}
   Let $\Sigma := \partial M$  and $h$ be a background Riemannian metric on $\Sigma$.  
For $\delta > 0$  consider
\begin{align}
\hat{M^\delta} = M \sqcup \Sigma \times [-\delta, \delta] \sqcup M.
\end{align}
Then there exists $\hat{\delta}>0$ such that for $\delta < \hat{\delta}$ there exist Riemannian metrics 
$g_\alpha^{\delta}$  on $\hat M^\delta$,  $\alpha \in \mathbb N \cup \{0\}$,  
so that  $g_\alpha^{\delta}=g_\alpha$ on both factors of $M$ inside $\hat{M}^\delta$.
If $(z,t)$ are coordinates on $\Sigma \times [-\delta,\delta]$ and $v,w$ are vector fields on $\Sigma$, the restricted metrics $h_{\alpha}^{\delta}:=g_{\alpha}^{\delta}|_{\Sigma \times \{t\}}$ satisfy 
\begin{align}\label{FToCalpha}
h_\alpha^{\delta}(v,w)(z,t) & = h_\alpha^{\delta}(v,w)(z,-\delta)+2\int_{-\delta}^t A_0^{\delta}(v,w)(z,s) ds,    
\end{align}
\begin{align}
h_\alpha^{\delta}(v,w)(z,t)  = h_\alpha^{\delta}(v,w)(z,-t) ,
\end{align}
and
\begin{align}\label{UniformMetricDeltaConvergence}
|h_\alpha^{\delta}(z,-\delta) - h_\alpha^{\delta}(z,t)|_{h} &  \le 4(m-1)C \delta, \quad \forall t \in[-\delta,\delta],
\end{align}
where $A_0^\delta$ is a smooth, symmetric tensor on $\Sigma \times [-\delta,\delta]$ which agrees smoothly with the second fundamental form of $\partial M \subset M_0$ at $\Sigma \times \{-\delta\}$, is an odd function in $t$
\begin{align}
A_0^{\delta}(z,t)= -A_0^{\delta}(z,-t),
\end{align}
and has bounded norm as a tensor on $\Sigma \times \{t\}$
\begin{align}
|A_0^{\delta}|_h \le C.\label{ABound1}
\end{align}
Furthermore, $g_0^\delta$ is a smooth Riemannian metric while $g_j^\delta$, $j \in \N$ are continuous Riemannian metrics. $\hat{\delta} = \hat{\delta}(C,g_0|_{\Sigma})$ and there exists a $\eta=\eta(\hat{\delta})>0$ so that 
\begin{align}
h_0^{\delta}(v,v)(z,t) \ge \eta h(v,v), \quad \forall v \in T\Sigma, (z,t) \in \Sigma \times [-\delta,\delta],
\end{align}
for all $\delta < \hat{\delta}$.
\end{thm}

\begin{proof}
Define $\hat{M}^\delta_\alpha=(\hat{M}^\delta,g_\alpha^{\delta})$ where $g_\alpha^{\delta} = g_\alpha$ on both factors of $M$ inside $\hat{M}^\delta$.
Let $(z,t)$ be coordinates on $\Sigma \times [-\delta,\delta]$ then we define
\begin{align}
g_\alpha^{\delta}(\partial_t,\partial_t)&=1,
\\g_\alpha^{\delta}(\partial_t,\partial_{z_i})&=0 \quad 1 \le i \le m-1.
\end{align}

To define $g_\alpha^{\delta}(\partial_{z_i},\partial_{z_k})$ we proceed as follows. 
Let $\Sigma_t=\Sigma \times \{t\}$ then  we write $g_\alpha^{\delta}|_{\Sigma_t} = h_\alpha^{\delta}(z_1,...,z_{m-1},t)$. 
For $\alpha=0$,  $h_0^{\delta}$  is a solution to the equation
\begin{align}
\frac{\partial h_0^{\delta}}{\partial t}(z,t)  &=2A_0^{\delta}(z,t),\label{MetricEvEq}
\\  h_0^{\delta}(z,-\delta) &= g_0|_{\Sigma}, 
\end{align}
where  $A_0^{\delta}(z,t)$ denotes the second fundamental form of $\Sigma_t$ and is 
chosen to be smooth, agree smoothly with the second fundamental form of $\Sigma \subset M_0$, odd in $t$
\begin{align}
A_0^{\delta}(z,t) = -A_0^{\delta}(z,-t),\label{ASymmetry} 
\end{align}
and uniformly bounded in $\delta$
\begin{align}
|A_0^{\delta}|_{h} \le C,\label{ABounded}
\end{align}
so that $\hat{\delta}$ can be chosen so that $h_0^{\delta}$ is positive definite on $\Sigma \times [-\delta,\delta]$ for $\delta \le \hat{\delta}$.
Notice that the fact that we can extend $A_0^{\delta}$
to be smooth implies that $h_0^{\delta}$ is a smooth metric. 

Now we just need to discuss how to attach the Riemannian neck in a smooth way to $M$. For this let $\nu_0$ be the inward pointing unit normal vector to $\partial M = \Sigma$ with respect to $g_0$.  Then let $F: \Sigma\times [0,\varepsilon) \rightarrow M$ be a one parameter family of smooth maps which solve
\begin{align}
\frac{\partial F}{\partial s} = \nu_0
\end{align}
which we know has a solution for at least a short time $s \in [0,\varepsilon_0)$. Now by the change of coordinates $t=-\delta-s$ we let $\Sigma_{-\delta-s}:= F(\Sigma,s)$ so that $\Sigma \times (-\delta-\varepsilon_0,-\delta]$ fits smoothly into the coordinates defined on $\Sigma \times [-\delta,-\delta]$. See the proof of Theorem 9 in Bray \cite{Bray-Penrose} for a similar discussion of the construction above.

For $\alpha=j \in \mathbb N$,  $h_j^{\delta}$ is the solution to the equation
\begin{align}\label{eq-gdeltaj}
\frac{\partial h_j^{\delta}}{\partial t}(z,t)  &=2A_0^{\delta}(z,t),
\\  h_j^{\delta}(z,-\delta) &= g_j|_{\Sigma}.
\end{align}
To attach the metric to $M_j$ we first note by Lemma \ref{RealignCoordsLem} we may assume that $\nu_0$ is also the inward unit normal vector with respect to $g_j$. Hence we identify $\partial_t$ on $\Sigma_{-\delta}$ with $-\nu_0$, the outward unit normal vector to $\Sigma \subset M$ with respect to $g_j$. Then since the outward normal vectors to the boundary $\Sigma$ agree with respect to $g_{\alpha}$ this ensures that the metric comparison on $M$ and the metric comparison on $\Sigma \times [-\delta,\delta]$ are compatible at the boundary so that we will be able to conclude $g_j^{\delta} \ge g_0^{\delta}$. In this case we note that $g_j^{\delta}$ will not be a smooth metric but instead just continuous due to the incompatibility of the second fundamental form of $\Sigma$ from inside $M$ versus inside the neck $\Sigma \times [-\delta,\delta]$.

For $v,w$ vector fields on $\Sigma$ we observe
\begin{align}
h_\alpha^{\delta}(v,w)(z,t)-h_\alpha^{\delta}(v,w)(z,-\delta) & = \int_{-\delta}^t \frac{\partial}{\partial s} \left(h_{\alpha}^{\delta}(v,w)(z,s) \right) ds,
\end{align}
and since the vector fields do not depend on time we have
\begin{align}\label{halphadelta}
h_\alpha^{\delta}(v,w)(z,t) & = h_\alpha^{\delta}(v,w)(z,-\delta)+2\int_{-\delta}^t A_0^{\delta}(v,w)(z,s) ds.
\end{align}
Since   $h_j^{\delta}(z,-\delta) =  h_j|_{\Sigma_\delta}(z)   \ge  h_0|_{\Sigma_\delta}(z) = h_0^{\delta}(z,-\delta) $
and $h_0^{\delta}(z,t)$ is positive definite on $\Sigma\times [-\delta,\delta]$ then $h_j^{\delta}(z,t)$ is positive definite on $\Sigma\times [-\delta,\delta]$. 

Note that by the oddness of $A_0^{\delta}$ we know
\begin{align}
\int_{-t}^t A_0^{\delta}(v,w)(z,s) ds = 0
\end{align} 
and hence
\begin{align}
h_\alpha^{\delta}(v,w)(z,-t) & = h_\alpha^{\delta}(v,w)(z,-\delta)+2\int_{-\delta}^{-t} A_0^{\delta}(v,w)(z,s) ds
\\&=h_\alpha^{\delta}(v,w)(z,-\delta)+2\int_{-\delta}^{t} A_0^{\delta}(v,w)(z,s) ds
\\&=h_\alpha^{\delta}(v,w)(z,t).
\end{align}
This symmetry is used to attach the second copy of $M$ to the other end of the neck $\Sigma \times [-\delta,\delta]$.

Now we observe that  for  an $h$ orthornormal frame  $\{e_1,...,e_{m-1}\}$ on $\Sigma$ we can combine  \eqref{ABounded} and  \eqref{halphadelta} to obtain the estimate
\begin{align}
h_\alpha^{\delta}(e_i,e_k)(z,-\delta) - 2C(t+\delta) \le & h_\alpha^{\delta}(e_i,e_k)(z,t) \\
\le & h_\alpha^{\delta}(e_i,e_k)(z,-\delta) + 2C(t+\delta). \label{LipBounds}
\end{align}
From the previous inequality we can estimate
\begin{align}
|h_\alpha^{\delta}(z,-\delta) - h_\alpha^{\delta}(z,t)|_{h} &= \sqrt{\sum_{i,k=1}^{m-1} \left(h_\alpha^{\delta}(z,-\delta)(e_i,e_k) - h_\alpha^{\delta}(z,t)(e_i,e_k)\right)^2}
\\&\le \sqrt{4C^2(t+\delta)^2(m-1)^2}
\\&\le 4(m-1)C \delta \quad \forall t \in[-\delta,\delta].
\end{align}

If we choose $\{w_1,...,w_{m-1}\}$ an orthonormal basis with respect to $h$ at $z \in \Sigma$ which diagonalizes $A_0^{\delta}(z,t)$ and express $v \in T_z\Sigma$ as $\displaystyle\sum_{i=1}^{m-1}v_iw_i$ then by applying \eqref{ABound1} we find
\begin{align}
A_0^{\delta}(v,v)(z,t) = \sum_{i=1}^{m-1} v_i^2 A_0^{\delta}(w_i,w_i)(z,t) \ge - C  h(v,v). 
\end{align}
Since this argument can be repeated for every $(z,t) \in \Sigma\times [-\delta,\delta]$ we find
\begin{align}
A_0^{\delta}(v,v)(z,t) \ge -C h(v,v), \quad \forall (z,t) \in \Sigma \times [-\delta,\delta], v \in T_z\Sigma.
\end{align}
So by \eqref{halphadelta} we see that
\begin{align}\label{h0deltaposdef}
 h_0^{\delta}(v,v)(z,t) \ge h_0^{\delta}(v,v)(z,-\delta) - 4\delta C h(v,v) 
\end{align}
and hence by choosing $\hat{\delta}$ small enough so that
\begin{align}
h_0^{\delta}(v,v)(z,-\delta)=h_0(v,v) \ge 8 \hat{\delta} C h(v,v), \quad \forall v \in T\Sigma,
\end{align}
we can choose $\eta = 4\hat{\delta} C$ so that for $\delta < \hat{\delta}$, 
\begin{align}
 h_0^{\delta}(v,v)(z,t) &\ge h_0^{\delta}(v,v)(z,-\delta) - 4\delta C h(v,v)
 \\ &\ge  h_0(v,v) - 4\hat{\delta} C h(v,v)
 \\ &\ge 8 \hat{\delta} C h(v,v)- 4\hat{\delta} C h(v,v) = 4\hat{\delta} C h(v,v) = \eta h(v,v). \label{h0deltaposdef2}
\end{align}
\end{proof}

Using the fact that the metrics $g_\alpha^\delta$  equal $g_\alpha$ inside $M$ and the
estimate they satisfy in the neck regions \eqref{UniformMetricDeltaConvergence}, we prove that 
the difference of the distances $d_\alpha^\delta$ and $d_\alpha$  for points inside $M$ are uniformly bounded in terms of $\delta$. 

\begin{thm}\label{DeltaToZeroThm}
Let $M^m$ be a compact, oriented, and connected manifold with non empty boundary, $(M,g_j)$ a sequence of continuous Riemannian manifolds, $(M,g_0)$ a smooth Riemannian manifold, and $h$ a smooth background Riemannian manifold on $\Sigma:= \partial M$. Assume that
\begin{align}
g_0(v,v) < g_j(v,v), \quad \forall p \in M, v \in T_pM, v \not = 0.
\end{align}
Let $\hat{M}_0^{\delta}=(\hat{M}^\delta,g_0^{\delta})$ and $\hat{M}_j^{\delta}=(\hat{M}^\delta,g_j^{\delta})$ be given as in Theorem \ref{FamilyOfMetrics}.  Then for any  $\alpha \in \mathbb N \cup \{0\}$, $\delta < \hat{\delta}=\hat{\delta}(C,g_0|_{\Sigma})$ as in Theorem \ref{FamilyOfMetrics}, and $p,q \in M \subset \hat{M}^{\delta}$ we find
\begin{align}
|d_{\alpha}(p,q)-d_{\alpha}^{\delta}(p,q)| \le 2\eta^{-1} \sqrt{C \delta}\diam(M_{\alpha}),
\end{align}
where $\eta=\eta(\hat{\delta}) >0$ is as in Theorem \ref{FamilyOfMetrics} so that 
\begin{align}
\eta h(v,v) \le h_0^{\delta}(v,v)(z,t), \quad \forall v \in T_z \Sigma, (z,t) \in \Sigma \times [-\delta,\delta].
\end{align}
\end{thm}

\begin{proof}
Consider points $p,q \in M \subset \hat{M}^{\delta}$ and assume that $M$ is the part of the $\delta$-doubling $\hat M^\delta$ closer to $\Sigma_{-\delta}$.    Let $\gamma_\alpha^{\delta}(s)\subset \hat{M}^{\delta}$ be a curve which almost minimizes the distance between $p,q$ with respect to $g_\alpha^{\delta}$, i.e.
\begin{align}
d_\alpha^{\delta}(p,q)+\varepsilon &\ge L_{g_\alpha^{\delta}}(\gamma_\alpha^{\delta}). 
\end{align}
If $\gamma_\alpha^{\delta}\subset M \subset \hat{M}^\delta$ then there is nothing to argue. Notice that by the symmetry of $\hat{M}_{\alpha}^{\delta}$ it is not efficient for $\gamma_{\alpha}^{\delta}$ to enter the second copy of $M \subset \hat{M}^{\delta}$, so in this case we can decompose the curve into three pieces $\gamma_\alpha^{\delta}= \gamma_\alpha^{\delta,1}\gamma_\alpha^{\delta,2}\gamma_\alpha^{\delta,3}$ where $\gamma_\alpha^{\delta,1},\gamma_\alpha^{\delta,3} \subset M$ and $ \gamma_\alpha^{\delta,2} \subset \Sigma \times [-\delta,\delta]$. Also, define $\bar{\gamma}_\alpha^{\delta}= \gamma_\alpha^{\delta,1}\bar{\gamma}_\alpha^{\delta,2}\gamma_\alpha^{\delta,3}$ where $\bar{\gamma}_\alpha^{\delta,2} \subset \Sigma_{-\delta}$ is connecting the endpoints of $\gamma_\alpha^{\delta,1},\gamma_\alpha^{\delta,3}$ so that if $\gamma_{\alpha}^{\delta,2}(s)=(z_{\alpha}^{\delta,2}(s),t_{\alpha}^{\delta,2}(s))$ then we choose $\bar{\gamma}_\alpha^{\delta,2}$ so that
\begin{align}
L_{g_{\alpha}}(\bar{\gamma}_\alpha^{\delta,2})&\le L_{g_{\alpha}|_{\Sigma}}(z_{\alpha}^{\delta,2}).\label{lengthBoundaryInequality}
\end{align}

By assumption we note that
\begin{align}
d_{\alpha}^{\delta}(p,q)+\varepsilon &\ge L_{g_\alpha^{\delta}}(\gamma_\alpha^{\delta}),
\\d_\alpha(p,q) &\le L_{g_\alpha}(\bar{\gamma}_\alpha^{\delta}).
\end{align}
In addition, by construction
\begin{align}
 L_{g_\alpha^{\delta}}(\gamma_\alpha^{\delta}) &= L_{g_\alpha^{\delta}}(\gamma_\alpha^{\delta,1})+L_{g_\alpha^{\delta}}(\gamma_\alpha^{\delta,2})+L_{g_\alpha^{\delta}}(\gamma_\alpha^{\delta,3})
 \\&=L_{g_\alpha}(\gamma_\alpha^{\delta,1})+L_{g_\alpha^{\delta}}(\gamma_\alpha^{\delta,2})+L_{g_\alpha}(\gamma_\alpha^{\delta,3}),
\end{align}
and hence we are left to estimate $L_{g_\alpha^{\delta}}(\gamma_\alpha^{\delta,2})$.

Notice that on $\Sigma \times [-\delta,\delta]$ we can rewrite the metrics as
\begin{align}
g_\alpha^{\delta} = dt^2+h_\alpha^{\delta},
\end{align}
and we let $h=g|_{\Sigma}$, where $g$ is a background metric and $h_{\alpha}=g_{\alpha}|_{\Sigma} = h_{\alpha}^{\delta}(z,-\delta)$, 
where the last equality follows from the definition of $g^\delta_\alpha$. 

By combining \eqref{ABound1} with \eqref{FToCalpha} we find 
\begin{align}
 h_\alpha^{\delta}(v,v)(z,t) \ge h_\alpha^{\delta}(v,v)(z,-\delta) - 4C\delta h(v,v), \quad \forall (z,t) \in \Sigma \times [-\delta,\delta], v \in T_z\Sigma,
\end{align}
and hence for any $v \in T(\Sigma \times [-\delta,\delta])$
\begin{align}
dt(v)^2&+h_\alpha(dz(v),dz(v))-4C\delta h(dz(v),dz(v)) \le g_\alpha^{\delta}(dz(v),dz(v)),\label{MetricInequality}
\end{align}
where $dz$ projects onto $T\Sigma$.   Note that  for $\alpha=0$ the metric on the left side of \eqref{MetricInequality} is positive definite  by \eqref{h0deltaposdef}-\eqref{h0deltaposdef2}.    Since $h_0 \le h_{\alpha}$, 
it follows that  the metric on the left side of \eqref{MetricInequality} is positive definite for all $\alpha$ and  any $\delta < \hat{\delta}$.
Hence for $\gamma_{\alpha}^{\delta,2}(s)=(z_{\alpha}^{\delta,2}(s),t_{\alpha}^{\delta,2}(s))$ this implies
\begin{align}
 &L_{g_\alpha^{\delta}}(\gamma_\alpha^{\delta,2}) \nonumber
 \\= &  \int_0^1 \sqrt{g_\alpha^{\delta}((\gamma_\alpha^{\delta,2})',(\gamma_\alpha^{\delta,2})')}dt \nonumber \\
 \ge &  \int_0^1 \sqrt{((t_\alpha^{\delta,2})')^2+h_\alpha((z_\alpha^{\delta,2})',(z_\alpha^{\delta,2})')-4C\delta h((z_\alpha^{\delta,2})',(z_\alpha^{\delta,2})')}dt  \nonumber \\
 \ge & \int_0^1 \sqrt{h_\alpha((z_\alpha^{\delta,2})',(z_\alpha^{\delta,2})')-4C\delta h((z_\alpha^{\delta,2})'),(z_\alpha^{\delta,2})')}dt. \label{eq-h-4c}
\end{align}

Now we want to estimate \eqref{eq-h-4c}.  Remember that by Theorem \ref{FamilyOfMetrics} we chose an $\eta=\eta(\hat{\delta}) > 0$ so that
\begin{align}
0<\eta h(v,v) \le h_0^{\delta}(v,v)(z,t) \quad \forall v \in T_z \Sigma, (z,t) \in \Sigma \times [-\delta,\delta].
\end{align}
Then recall that for $x \ge y$
\begin{align}
\sqrt{x^2-y^2} = \sqrt{(x+y)(x-y)} = \sqrt{x+y}\sqrt{x-y} \ge x-y. 
\end{align}

We now estimate \eqref{eq-h-4c}. This is done by noticing that $(z_\alpha^{\delta,2})'  \in  T \Sigma_{-\delta}$,  using the fact that the left hand side of \eqref{MetricInequality} is positive definite for $\delta < \hat{\delta}$ and that 
by \eqref{halphadelta}  $h_0^{\delta} \le h_j^{\delta}$ on $\Sigma \times [-\delta,\delta]$. 
\begin{align}
  L_{g_\alpha^{\delta}}(\gamma_\alpha^{\delta,2}) &\ge \int_0^1 \sqrt{h_\alpha((z_\alpha^{\delta,2})',(z_\alpha^{\delta,2})')}dt
  \\& \quad-\int_0^1\sqrt{4C\delta h((z_\alpha^{\delta,2})',(z_\alpha^{\delta,2})')}dt
  \\&=  L_{h_{\alpha}}(z_{\alpha}^{\delta,2}) - 2\sqrt{C \delta} \int_0^1\sqrt{ h((z_\alpha^{\delta,2})',(z_\alpha^{\delta,2})')}dt
  \\&\ge L_{g_\alpha}(\bar{\gamma}_\alpha^{\delta,2}) - 2\eta^{-1}\sqrt{C \delta} \int_0^1\sqrt{((t_{\alpha}^{\delta,2})')^2+ h_{\alpha}^{\delta}((z_\alpha^{\delta,2})',(z_\alpha^{\delta,2})')}dt
  \\&= L_{g_\alpha}(\bar{\gamma}_\alpha^{\delta,2}) - 2\eta^{-1}\sqrt{C \delta} L_{g_{\alpha}^{\delta}}(\gamma_{\alpha}^{\delta,2})
  \\&\ge L_{g_\alpha}(\bar{\gamma}_\alpha^{\delta,2}) - 2\eta^{-1}\sqrt{C \delta} L_{g_{\alpha}^{\delta}}(\gamma_{\alpha}^{\delta})
  \\&\ge L_{g_\alpha}(\bar{\gamma}_\alpha^{\delta,2}) - 2\eta^{-1}\sqrt{C \delta} (d_{\alpha}^{\delta}(p,q)+\varepsilon)
  \\&\ge L_{g_\alpha}(\bar{\gamma}_\alpha^{\delta,2}) - 2\eta^{-1}\sqrt{C \delta} (d_{\alpha}(p,q)+\varepsilon)\label{SecondToLastLine}
  \\&\ge L_{g_\alpha}(\bar{\gamma}_\alpha^{\delta,2}) - 2\eta^{-1}\sqrt{C \delta} (\diam(M_{\alpha})+\varepsilon).
\end{align}
In the last part we used the fact that $d_{\alpha}(p,q) \ge d_{\alpha}^{\delta}(p,q)$ for $p,q \in M$ in \eqref{SecondToLastLine}.
Then we note
\begin{align}
d_\alpha(p,q) - d_\alpha^{\delta}(p,q) - \varepsilon &\le L_{g_\alpha}(\bar{\gamma}_\alpha^{\delta}) - L_{g_\alpha^{\delta}}(\gamma_\alpha^{\delta})
\\& =L_{g_\alpha}(\bar{\gamma}_\alpha^{\delta,2})- L_{g_\alpha^{\delta}}(\gamma_\alpha^{\delta,2})
\\&\le 2\eta^{-1}\sqrt{C \delta} (\diam(M_{\alpha})+\varepsilon).
\end{align}
Since $d_\alpha(p,q) \ge d_\alpha^{\delta}(p,q)$, by construction, this implies
\begin{align}
|d_\alpha(p,q) - d_\alpha^{\delta}(p,q)| \le  2\eta^{-1}\sqrt{C \delta} (\diam(M_{\alpha})+\varepsilon)+\varepsilon,
\end{align}
 and since this is true for all $\varepsilon>0$ we find
 \begin{align}
|d_\alpha(p,q) - d_\alpha^{\delta}(p,q)| \le  2\eta^{-1}\sqrt{C \delta} \diam(M_{\alpha}),
\end{align}
 which is the uniformity needed in $\alpha$.
\end{proof}

By imposing a condition on the boundaries $\partial M_j, \partial M_0$ we can show that 
$\vol(M_j^\delta)  \to  \vol(M_0^\delta)$ and thus we can apply Allen-Sormani's result, Theorem \ref{PointwiseConvergenceAE},
to get almost everywhere convergence up to a subsequence of the distance functions, $d_j^\delta  \to d_0^\delta$.

\begin{thm}\label{jTo0WithDeltaThm}
Let $M^m$ be a compact, oriented, and connected manifold with non empty boundary, $(M,g_j)$ a sequence of continuous Riemannian manifolds,  $\partial M = \Sigma$, and $(M,g_0)$, $(\Sigma, h)$ smooth background Riemannian manifolds such that 
\begin{align}
    \vol(M_j) &\rightarrow \vol(M_0),
    \\ \|g_j|_{\Sigma}-g_0|_{\Sigma}\|&_{L^{\frac{m-1}{2}}(\partial M,h)}\rightarrow 0,\label{LpBound}
    \\ g_j(v,v) &> g_0(v,v) \quad \quad \forall p \in M, v \in T_pM, v \not = 0.
\end{align}
Let $\hat{M}_0^{\delta}=(\hat{M}^\delta,g_0^{\delta})$ and $\hat{M}_j^{\delta}=(\hat{M}^\delta,g_j^{\delta})$ be given as in Theorem \ref{FamilyOfMetrics}.
 
Then for a subsequence 
\begin{align}
d_{j(k)}^{\delta}(p,q) \rightarrow d_0^{\delta}(p,q)  \,\, \text{as}  \,\, k \to \infty
\end{align}
 for  $\dvol_{g_0^\delta} \times \dvol_{g_0^\delta}$   a.e. $(p,q) \in \hat{M}^\delta \times  \hat{M}^\delta$ and 
\begin{align}
\hat{M}_j^{\delta} &\VFto \hat{M}_0^{\delta}. 
\end{align}
\end{thm}

\begin{proof}
By using $A_0^{\delta}$ to define all the metrics $h_\alpha^\delta$  on $\Sigma \times [-\delta,\delta]$ and the assumption that $g_0 \le g_j$,  we find for a vector field $v$ on $\Sigma$
\begin{align}
h_j^{\delta}(v,v)(z,t) & = h_j^{\delta}(v,v)(z,-\delta)+2\int_{-\delta}^t A_0^{\delta}(v,v)(z,s) ds
\\&\ge h_0^{\delta}(v,v)(z,-\delta)+2\int_{-\delta}^t A_0^{\delta}(v,v)(z,s) ds 
\\&= h_0^{\delta}(v,v)(z,t).
\end{align}
Then since any $w \in T_pM$, $p=(z,t) \in \Sigma \times [-\delta,\delta]$, can be written as $(v,c \partial_t)$, $v \in T_z \Sigma$ and $g_{\alpha}^{\delta}(v,\partial_t)=0$, we find
\begin{align}
g_0^{\delta}(w,w) &= c^2g_0^{\delta}(\partial_t,\partial_t) + g_0^{\delta}(v,v)\label{MetricLowerBoundDelta1} 
\\&= c^2 + h_0^{\delta}(v,v)  
\\&\le c^2 + h_j^{\delta}(v,v) 
\\ &= c^2g_j^{\delta}(\partial_t,\partial_t) + g_j^{\delta}(v,v) = g_j^{\delta}(w,w),\label{MetricLowerBoundDelta2}
\end{align}
and hence
\begin{align}
g_0^{\delta} \le g_j^{\delta}.\label{MetricLowerBoundDelta}
\end{align}

For the volume estimate we use \eqref{FToCalpha} and assume that $\{e_1,...,e_{m-1}\}$ is an orthonormal frame for $h$ on $\Sigma$ to estimate
\begin{align}
|h_j^{\delta}(z,t)  -h_0^{\delta}(z,t) |_h = &  \sqrt{\sum_{k,l=1}^{m-1} \left(h_j^{\delta}(z,t)(e_k,e_l) - h_0^{\delta}(z,t)(e_k,e_l)\right)^2}\\
= & \sqrt{\sum_{k,l=1}^{m-1} \left(h_j^{\delta}(z,-\delta)(e_k,e_l) - h_0^{\delta}(z,-\delta)(e_k,e_l)\right)^2}\\
= &  |h_j^{\delta}(z,-\delta)-h_0^{\delta}(z,-\delta) |_h
\end{align}
which by combining with the assumption \eqref{LpBound} implies 
\begin{align}
&\int_{\Sigma}|h_j^{\delta}(z,t)-h_0^{\delta}(z,t)|_h^{\frac{m-1}{2}}dA_h\label{LpEstimate} 
\\&= \int_{\Sigma}|h_j^{\delta}(z,-\delta)-h_0^{\delta}(z,-\delta) |_h^{\frac{m-1}{2}}dA_h \rightarrow 0.
\end{align}
It was observed in a previous paper of the first named author and Sormani that $L^{\frac{m-1}{2}}$ convergence of a metric combined with a metric lower bound implies convergence of areas (See Lemma 2.7 and Lemma 4.3 of \cite{Allen-Sormani-2}) and hence we find
\begin{align}
\area(\Sigma,h_j^{\delta}(\cdot,t)) \rightarrow \area(\Sigma,h_0^{\delta}(\cdot,t)), \quad \forall t \in [-\delta,\delta].
\end{align}
This allows us to compute
\begin{align}
&|\vol_j(\Sigma\times[-\delta,\delta])-\vol_0(\Sigma\times[-\delta,\delta])|
\\&= \left|\int_{-\delta}^{\delta} \area(\Sigma,h_j^{\delta}(\cdot,t))dt-\int_{-\delta}^{\delta}\area(\Sigma,h_0^{\delta}(\cdot,t))dt\right|
\\&\le \int_{-\delta}^{\delta} \left|\area(\Sigma,h_j^{\delta}(\cdot,t))-\area(\Sigma,h_0^{\delta}(\cdot,t))\right|dt \rightarrow 0,
\end{align}
where we use the dominated convergence theorem in the last line since $\left|\area(\Sigma,h_j^{\delta}(\cdot,t))-\area(\Sigma,h_0^{\delta}(\cdot,t))\right|$ is bounded for all $t \in [-\delta,\delta]$ by \eqref{LpBound} and \eqref{LpEstimate}.

Since we have assumed $\vol(M_j) \rightarrow \vol(M_0)$ and we have shown that the volume of the collar region converges as well we conclude that 
\begin{align}
\vol(\hat{M}_j^{\delta}) \rightarrow \vol(\hat{M}_0^{\delta}).\label{VolumeConvergenceDelta}
\end{align}

Now recall that $(\hat M^\delta, g_0^\delta)$ is a smooth Riemannian manifold, this combined with
\eqref{MetricLowerBoundDelta}, and \eqref{VolumeConvergenceDelta} allow us to apply 
Theorem \ref{PointwiseConvergenceAE}   to get a subsequence such that 
\begin{align}
d_{j(k)}^{\delta}(p,q) \rightarrow d_0^{\delta}(p,q)
\end{align}
for  $\dvol_{g_0^\delta}  \times \dvol_{g_0^\delta}$  a.e. $(p,q)  \in \hat M^\delta  \times \hat M^\delta$.    By calculating an upper bound for $\diam(\hat M_j^\delta)$ we can apply 
Theorem \ref{vol-thm} to get 
\begin{align}
\hat{M}_j^{\delta} \VFto \hat{M}_0^{\delta}. 
\end{align}
\end{proof}

Now we combine all the results from this section with a triangle inequality argument to obtain the desired conclusion.

\begin{thm}\label{thm-distdj}
Let $M^m$ be a compact, oriented, connected manifold with non empty boundary, $(M,g_j)$ a sequence of continuous Riemannian manifolds, $(M,g_0)$ a smooth Riemannian manifold, and $h$ a smooth background Riemannian manifold on $\Sigma:=\partial M$ such that 
\begin{align}
    \vol(M_j) &\rightarrow \vol(M_0),
    \\ \|g_j|_{\Sigma}-g_0|_{\Sigma}\|&_{L^{\frac{m-1}{2}}(\Sigma,h)}\rightarrow 0,
    \\ \diam(M_j) &\le D,
    \\ g_j(v,v) &> g_0(v,v) \quad \forall v \in T_pM.
\end{align}
Then for a subsequence we have 
\begin{align}
d_{j(k)}(p,q) \rightarrow d_0(p,q)
\end{align}
for  $\dvol_{g_0} \times \dvol_{g_0}$  a.e. $(p,q) \in M  \times M$.
\end{thm}

\begin{proof}
Let $\delta_i$ be a sequence of real numbers decreasing to zero so that $\delta_i < \hat{\delta}$ where $\hat{\delta}$ is from Theorem \ref{FamilyOfMetrics}. 
Then apply Theorem  \ref{FamilyOfMetrics} to find Riemannian manifolds
$(M_j^{\delta_i}, g_j^{\delta_i})$ and $(M_0^{\delta_i}, g_0^{\delta_i})$. 
Then   for $p,q \in M \subset \hat{M}^{\delta_i}$, by the triangle inequality, we find
\begin{align}
|d_j(p,q)-d_0(p,q)| &\le |d_j(p,q) - d_j^{\delta_i}(p,q)|
\\&\quad+|d_j^{\delta_i}(p,q)-d_0^{\delta_i}(p,q)|+|d_0^{\delta_i}(p,q)-d_0(p,q)|.
\end{align}
If we apply the estimates of Theorem \ref{DeltaToZeroThm} we find  
\begin{align}
|d_j(p,q)-d_0(p,q)| &\le |d_j^{\delta_i}(p,q)-d_0^{\delta_i}(p,q)|+2\eta^{-1}\sqrt{C \delta_i} D.
\end{align}
Now by Theorem \ref{jTo0WithDeltaThm} for each $i$ there is a subsequence such that 
\be
 |d_{j(i,k)}^{\delta_i}(p,q)-d_0^{\delta_i}(p,q)|   \to  0
\ee
for  $\dvol_{g_0^{\delta_i}} \times \dvol_{g_0^{\delta_i}}$   a.e. $(p,q) \in \hat{M}^{\delta_i} \times  \hat{M}^{\delta_i}$.
Notice that we are only considering points $(p,q)  \in M \times M$ and there we have $\dvol_{g_0^{\delta_i}} \times \dvol_{g_0^{\delta_i}}=\dvol_{g_0} \times \dvol_{g_0}$ since $g^\delta_0=  g_0$ in both copies of $M$. 

By a diagonalization process we can assume that $j(k)=j(i,k)=j(i',k)$ for all $i, i'$. 
Thus,  we can take the limit 
\begin{align}
\limsup_{k\rightarrow \infty}|d_{j(k)}(p,q)-d_0(p,q)| &\le 2\eta^{-1}\sqrt{C \delta_i} D,
\end{align}
then we let $i  \to \infty$, 
\begin{align}
\limsup_{k\rightarrow \infty}|d_{j(k)}(p,q)-d_0(p,q)| &= 0.
\end{align}

\end{proof}


\section{Proof of Main Theorems}\label{Proofs}

We are ready to prove our main results. We start by proving the following theorem and then use it to prove the main theorems of the introduction.
 
\begin{thm}\label{PrelimMainThm}
Let $M$ be a compact, oriented, and connected manifold.  Let $M_0=(M,g_0)$  be a smooth Riemannian manifold and $M_j = (M,g_j)$ a sequence of continuous Riemannian manifolds
such that
\be 
g_0(v,v) < g_j(v,v)   \quad \forall v \in T_pM,
\ee
\be
\diam(M_j) \le D, 
\ee
\be 
\vol(M_j) \rightarrow \vol(M_0)
\ee
and
\be
 \|g_j|_{\partial M}-g_0|_{\partial M}\|_{L^{\frac{m-1}{2}}(\partial M,h)}\rightarrow 0.
\ee
Then
\begin{align}
M_j \VFto M_0. 
\end{align}
\end{thm}

\begin{proof}
We first apply Theorem \ref{thm-distdj} to  get a subsequence of $(M_j,g_j)$
that satisfies 
\be
d_{j(k)}(p,q) \to d_0(p,q)  \qquad \dvol_0 \times \dvol_0 \textrm{ a.e. } (p,q). 
\ee
Then by $ \|g_j|_{\partial M}-g_0|_{\partial M}\|_{L^{\frac{m-1}{2}}(\partial M,h)}\rightarrow 0$  we get 
$\vol(\partial M_j) \leq A$.   So now we can apply Theorem (\ref{thm-IFconvC0B}) to conclude that  
\begin{align}
M_{j(k)} \VFto M_0  \,\, \text{as} \,\, k \to \infty. 
\end{align}

To show that the whole sequence converges we proceed by contradiction. 
Assume that $d_{\mathcal F}(M_{j'(k)}, M_0) \geq \vare>0$  for some subsequence. 
By running the argument from the previous paragraph, there exists a subsequence of $M_{j'(k)}$  
that converges in intrinsic flat sense to $M_0$. This contradicts our  hypothesis. 
\end{proof}

\begin{proof}[  Proof of Theorem \ref{vol-thm-boundary}]
Consider $\tilde{g}_j = \frac{1}{1-\frac{1}{2j}} g_j$ and $\tilde{M}_j=(M,\tilde{g}_j)$. 
Then by reindexing the sequence if necessary we may assume that $\left(1 - \tfrac{1}{2j} \right)g_0 < g_j$ and hence $g_0   <  \tilde{g}_j$,
\begin{align}
\diam(\tilde{M}_j) = & \left(1-\tfrac{1}{2j} \right)^{\frac{1}{2}} \diam(M_j) \le D,\\
\vol(\tilde{M}_j) = & \left(1-\tfrac{1}{2j} \right)^{\frac{m}{2}} \vol(M_j) \rightarrow \vol(M_0)\\
\end{align}
and 
\be
\| \tilde g_j|_{\partial M}-g_0|_{\partial M}\|_{L^{\frac{m-1}{2}}(\partial M,h)}   \rightarrow   0.
\ee
Hence $\tilde{M}_j$ satisfies the hypotheses of Theorem \ref{PrelimMainThm} which implies
\begin{align}
\tilde{M}_j \VFto M_0.
\end{align}

On the other hand, by construction we have  $\|g_j-\tilde{g}_j\|_{C^0_{g_0}(M)} \rightarrow 0$ which implies 
\begin{align}
\sup_{p,q\in M}|d_j(p,q)-\tilde{d}_j(p,q)|\rightarrow 0,
\end{align}
and since $\tilde{g}_j \ge g_j$ we can apply Theorem \ref{est-SWIF} with $W_j=M$ to find
\begin{align}
d_{\mathcal{F}}(\tilde{M}_j, M_j) \rightarrow 0.
\end{align}
Hence, by the triangle inequality for the intrinsic flat distance we find
\begin{align}
M_j \VFto M_0.
\end{align}
\end{proof}

\begin{proof}[Proof of Theorem \ref{vol-thm-boundary-improved}]
Consider $\tilde{g}_j = \frac{1}{1-\frac{1}{2j}} g_j$ and $\tilde{M}_j=(M,\tilde{g}_j)$. 
Then by reindexing the sequence if necessary we may assume that $\left(1 - \tfrac{1}{2j} \right)g_0 < g_j$ and hence $g_0   <  \tilde{g}_j$. Now we note that it was observed in Lemma 2.7 and Lemma 4.3 of \cite{Allen-Sormani-2} that $\tilde{g}_j \ge g_0$ combined with
\be
\int_{M}|\tilde{g}_j-g_0|_{g_0}^{\frac{m}{2}} dV_{g_0} \rightarrow 0,
\ee
implies
\begin{align}
\vol(\tilde{M}_j) \rightarrow \vol(M_0).
\end{align}
Now let $\nu_0$ be the inward pointing unit normal vector to $\partial M = \Sigma$ with respect to $g_0$. Then let $F: \Sigma\times [0,t) \rightarrow M$ be a one parameter family of smooth maps which solve
\begin{align}
\frac{\partial F}{\partial t} = \nu_0
\end{align}
which we know has a solution for at least a short time $t \in [0,\varepsilon_0)$. Now let $\Sigma_{t}= F(\Sigma,t)$ and note that by the coarea formula
\begin{align}
\int_{M} |\tilde{g}_j-g_0|_{g_0}^{\frac{m}{2}} dV_{g_0}&\ge \int_0^{\varepsilon_0} \int_{\Sigma_t} |\tilde{g}_j-g_0|_{g_0}^{\frac{m}{2}} dA_{g_0} dt \rightarrow 0.
\end{align}
Hence we can choose a subsequence $\tilde{g}_{j(k)}$ so that for almost ever $t \in [0,\varepsilon_0)$ we find
\begin{align}
\int_{\Sigma_t} |\tilde{g}_{j(k)}-g_0|_{g_0}^{\frac{m}{2}} dA_{g_0} \rightarrow 0.
\end{align}
Now by H\"{o}lder's inequality we find
\begin{align}
\int_{\Sigma_t} |\tilde{g}_{j(k)}-g_0|_{g_0}^{\frac{m-1}{2}} dA_{g_0} \rightarrow 0.
\end{align}
Since we know that on $\Sigma_t$
\begin{align}
|\tilde{g}_{j(k)}-g_0|_{g_0} \ge |\tilde{g}_{j(k)}|_{\Sigma_t}-g_0|_{\Sigma_t}|_{g_0|_{\Sigma_t}},
\end{align}
we also find
\begin{align}
\int_{\Sigma_t} |\tilde{g}_{j(k)}|_{\Sigma_t}-g_0|_{\Sigma_t}|_{g_0|_{\Sigma_t}}^{\frac{m-1}{2}} dA_{g_0} \rightarrow 0.\label{BoundaryConvergence}
\end{align}

Let $t_i \in [0,\varepsilon_0)$ be a sequence of times so that \eqref{BoundaryConvergence} holds and so that $t_i \rightarrow 0$ as $i \rightarrow \infty$. Let $M_t \subset M$ be a maximal subset so that $\partial M_t = \Sigma_t$, $\Sigma \not \subset M_t$, and $\tilde{M}_{\alpha}^t \subset \tilde{M}_{\alpha}$ the Riemannian manifold with the restricted Riemannian metric where $\tilde{M}_0=M_0$ for notational convenience. Then we define $\tilde{d}_{\alpha}^t$ to be the distance function for the Riemannian manifold $\tilde{M}_{\alpha}^t$. Note that for all $p,q \in M_t$
\begin{align}
\tilde{d}_{\alpha}^t(p,q) \ge \tilde{d}_{\alpha}(p,q) \ge d_0(p,q),\label{DistanceInequality}
\end{align}
and also by Theorem \ref{thm-distdj} for each $i$ there exists a $j(i,k)$ so that
\begin{align}
\tilde{d}_{j(i,k)}^{t_i}(p,q) \rightarrow d_0(p,q)
\end{align}
for almost every $p,q \in \tilde{M}_{t_i}$. By a diagonalization process we can assume that $j(k)=j(i,k)=j(i',k)$ for all $i,i'$ and hence by \eqref{DistanceInequality} we find
\begin{align}
\tilde{d}_{j(k)}(p,q) \rightarrow d_0(p,q)
\end{align}
for almost every $p,q \in M$.

By assumption $\vol(\partial \tilde{M}_j) \leq A$ and so now we can apply Theorem (\ref{thm-IFconvC0B}) to conclude that  
\begin{align}
\tilde{M}_{j(k)} \VFto M_0  \,\, \text{as} \,\, k \to \infty. 
\end{align}

To show that the whole sequence converges we proceed by contradiction. 
Thus, assume that $d_{\mathcal F}(\tilde{M}_{j'(k)}, M_0) \geq \vare>0$  for some subsequence. 
We can run the argument of the previous paragraph, hence a subsequence of $\tilde{M}_{j'(k)}$  
converges in intrinsic flat sense to $M_0$. This contradicts our 
hypothesis. We conclude that $M_j \VFto M_0$ as well since $\tilde{g}_j-g_j \rightarrow 0$ as tensors.
\end{proof}


\section{Application to Intrinsic Flat Stability of the PMT for Graphs} \label{HLSapplication}

The Positive Mass Theorem of Schoen-Yau and later Witten~\cite{Schoen-Yau-positive-mass, Witten-positive-mass} states that any complete asymptotically flat manifold of nonnegative scalar curvature has nonnegative ADM mass. 
Furthermore, if the ADM mass is zero, then the manifold must be Euclidean space. 

The intrinsic flat stability of the positive mass theorem was conjectured by Lee and Sormani \cite{LeeSormani1} and has been shown in the rotationally symmetric case \cite{LeeSormani1}, in the graph case \cite{HLS}, and in various other cases.   
In this section we  prove the stability result for graphical hypersurfaces  in $\E^{n+1}$, with empty boundary,  appearing in  \cite{HLS}  by applying our Theorem \ref{convBdry}.   The original proof in \cite{HLS} has some steps at the end that are difficult to follow and may require deep new theorems about integral current spaces to make them completely rigorous.  Here we avoid such complications by studying the sequence of manifolds themselves rather than their limit spaces.

We first recall that the spatial $n$-dimensional Schwarzschild manifold with boundary of ADM mass $m>0$ can be isometrically embedded into $\mathbb{E}^{n+1}$ as the graph of a smooth function $S_m :  \mathbb{E}^n \smallsetminus B((2m)^{1/(n-2)})   \to \R$, with minimal boundary, such that the boundary lies in the plane $\E^n\times \left\{0\right\}$. Explicitly, for $n=3,4$ we have 
\be\label{S_m-eqn}
	S_m(x) =\left\{
	\begin{array}{ll}
	 \sqrt{ 8m (|x|- 2m)}  &\mbox{ for } n =3\\
	 \sqrt{2m} \log \left( \frac{|x|}{\sqrt{2m}} + \sqrt{\frac{|x|^2}{2m} -1} \right) &\mbox{ for } n =4.
	\end{array}\right.
\ee
We now define the class of uniformly asymptotically flat graphical hypersurfaces of $\E^{n+1}$ with uniformly bounded depth and nonnegative
scalar curvature for which stability will be proven.

\begin{defn}\label{def:hypotheses}
For $n\ge3$, $r_0, \gamma, D>0$, and $\alpha<0$, define $\Gr$ to be the space of all smooth complete Riemannian manifolds of nonnegative scalar curvature, $(M^n,g)$, with empty boundary, that admit a smooth Riemannian isometric embedding $\Psi:M\to \E^{n+1}$ such that the image $\Psi(M)$ is the graph of a function $f\in C^\infty(\E^{n}) $:
\be
\Psi(M)=\left\{(x,f(x)): \,\, x\in \E^n \right\}
\ee
and for almost every $h$, the level set 
\be\label{cond6}
f^{-1}(h)\subset \E^n
\textrm{ is strictly mean-convex and outward-minimizing,}
\ee
where strictly mean-convex means that the mean curvature is strictly positive, and outward-minimizing means that any region of $\E^n$ that contains the region enclosed by $f^{-1}(h)$ must have perimeter at least as large as $\mathcal{H}^{n-1}( f^{-1}(h))$.

In addition we require uniform asymptotic flatness conditions:
\be\label{cond3}
 |D f| \le \gamma
\textrm{ for }|x|\ge r_0/2  \textrm{ and }
\lim_{x\to\infty} |D f| =0. 
\ee
If $n\ge 5$, we require that 
$f(x)$ approaches a constant as $x\to\infty$. 
If $n=3$ or $4$, we require that the graph is asymptotically Schwarzschild:
\be\label{cond5}
\exists \Lambda\in \R \textrm{ such that }
 \left|f(x) - (\Lambda+ S_m(|x|)) \right| \le \gamma |x|^\alpha 
 \textrm{ for } |x|\ge r_0.
\ee
Finally we require that the regions
\be
\Omega(r_0)=\Psi^{-1}(\overline{B(r_0)}\times\R)\,\,\,
\textrm{ and } \,\,\, \Sigma(r_0)=\partial\Omega(r_0)
\ee
have bounded depth
 \be\label{cond7}
 \Depth(\Omega(r_0), \Sigma(r_0)) =\sup\left\{d_g(p,\Sigma(r_0)): p\in \Omega(r_0)\right\}
 \le D.
 \ee
\end{defn}

The geometric conditions on the level sets in \eqref{cond6} are needed to apply the estimates in~\cite{Huang-Lee-Graph}.   We recall that under those conditions Huang-Lee 
proved a stability result for graphical hypersurfaces with respect to the Federer-Fleming's flat topology in $\E^{n+1}$ \cite{Huang-Lee-Graph}. 
However,  convergence with respect to the  flat topology does not necessarily imply convergence with respect to the intrinsic flat topology,  see Example 2.8 in \cite{HLS}. The outward minimizing property is also used to estimate volumes needed for the proof of Theorem~\ref{thm-main}.     
Equations \eqref{cond3} and \eqref{cond5} encode the asymptotic flatness condition of the manifolds. They follow from the natural but much stronger requirement that the functions $f$ to be uniformly asymptotically Schwarzschild up to first order.    Condition \eqref{cond7} prevents the possibility of ``arbitrarily deep gravity wells", hence, it  is used to show uniform diameter bounds.

Here we show that the preimages of the intersections
of the graph $\Psi_j(M_j)$ with the cylinder $\overline{B(r)} \times \R$ converge. 

\begin{thm}[Huang-Lee-Sormani Theorem 1.3  \cite{HLS}]\label{thm-main}   
Let $n\ge3$, $r_0, \gamma, D>0$, $\alpha<0$, and $r\ge r_0$.  Let $M_j \in\Gr$ be a sequence of manifolds 
with 
  \be
  m_{ADM}(M_j) \to 0.
  \ee
Then for any $r  \geq r_0$ we have
\be
\Omega_j(r)  \VFto  \overline{B(r)},
\ee
where $\overline{B(r)}  \subset \E^n$ is the ball of radius $r$ around the origin and  $\Omega_j(r)=\Psi_j^{-1}( \overline{B(r)}\times\R)$.
 \end{thm}

We prove the theorem by applying Theorem \ref{convBdry} and recalling the uniform diameter and area 
bounds and volume convergence proven in \cite{HLS}.

\begin{proof}
Let $g_0$ be the Riemmanian metric of Euclidean space and define
$\tilde g_j= g_0 +   df_j \otimes df_j$.   
Then it follows that $g_0  \leq  \tilde g_j$.
Note that each $M_j=(M,g_j)$ is Riemannian isometric to 
$ (\R^n,  \tilde g_j)$.

In Theorem 3.1 in \cite{HLS} it is shown that 
\be
\diam(\Omega_j(r)) \leq 2D + \pi r \sqrt{1+\gamma^2}
\ee
and 
\be 
\vol(\partial \Omega_j(r)) \leq  \omega_n r^{n-1}\sqrt{1+\gamma^2}.
\ee
Furthermore, in Corollary 4.4  in \cite{HLS}  it is shown that 
\be\label{limsupVolume}
\limsup_{j \to \infty}  \vol(\Omega_j(r)) = \vol(B(r)).
\ee
Since $g_0 \leq \tilde g_j$, we know $\vol(B(r))  \leq  \vol(\Omega_j(r))$ and hence by combining with \eqref{limsupVolume} we find 
\be
\lim_{j \to \infty}  \vol(\Omega_j(r)) = \vol(B(r)). 
\ee
Hence, since $M_j$ is Riemannian isometric to $(\R^n,\tilde{g}_j)$ we can apply Theorem \ref{convBdry} to conclude that 
\be
\Omega_j(r)  \VFto  \overline{B(r)}. 
\ee
\end{proof}

Applying Theorem~\ref{thm-main}  
one should be able to show pointed convergence as in Theorem 1.4 of \cite{HLS}. The only technical part is to find a convergent  sequence of points $p_j$. For that,  in \cite{HLS}  the sequence is chosen 
by arguing that $\Sigma_j(r_0)$ converges in Gromov-Hausdorff sense to $\partial B(r_0)$. 
Recall that if a sequence of spaces  $X_j$  converges in Gromov-Hausdorff sense to $X_\infty$ 
then any sequence $p_j \in X_j$ has a convergent subsequence to a point $p_\infty \in X_\infty$. 

For technical reasons, it is not possible to proceed in that way.   This is due to the fact that
in order to define a converging sequence of points  $p_j \in X_j$ one has to isometrically embed 
each $X_j$ in a bigger space $Z$ where the corresponding images converge in Hausdorff sense to 
the image of $X_\infty$.   But the spaces $X_j$ might not converge in the flat sense in  $Z$. Thus, sequences converging in 
$Z$ do not provide immediate information of sequences of points converging when one also has intrinsic flat convergence.  In a recent paper by the second named author with collaborators  details were given how to choose the sequence of points obtaining pointed convergence  as in Theorem 1.4 of \cite{HLS} and had extended our application to include manifolds with boundary \cite{HLP}.

\section{Appendix}\label{sec-app}

\smallskip 

In this appendix we give a detailed formal proof of Theorem \ref{est-SWIF}. For the reader unfamiliar with the technical details of integral current spaces the proof given in section \ref{VADBrev} is recommended.

\begin{proof}[2nd proof of Theorem \ref{est-SWIF}] 
Let 
\be
\psi_i:  U_i  \subset \R^m  \to \psi_i(U_i)\subset M
\ee  
be an oriented atlas of smooth charts of $M$.  Since these charts are diffeomorphisms, we can consider that they 
are biLipschitz maps when seen as maps from $(\R^m, d_{\R^m})$  to $(M, d_{g_j})$. 
They can also be restricted to $A_{ik} \subset U_i$   to ensure they have pairwise disjoint images as required
when considering them as rectifiable charts for $M_j$.  So 
\begin{align}\label{eq-canonicalT}
[[M_j]] = & \sum_{i,k}   \psi_{i\sharp}   [[1_{A_{ik}} ]]. 
\end{align}
Since $F_j: M_j\to M_0$ is biLipschitz,  then the functions 
\be
F_j\circ \psi_i: A_{ik}  \subset \R^m  \to F_j(\psi_i(A_{ik}))\subset M_0
\ee  
can be used as rectifiable charts for  $M_0$, 
\begin{equation}
[[M_0]]=  {F_j}_\sharp[[M_j]] =   \sum_{i,k} (F_j   \circ  \psi_i)_\sharp  [[1_{A_{ik}} ]].  
\end{equation} 
Let $\iota: [0,h_j]  \to [0,h_j]$ be the identity map. Then, 
\be
(\psi_i, \iota) :  A_{ik}\times [0,h_j]  \to    \psi_i(A_{ik}) \times [0,h_j]   \subset   M_j  \times [0,h_j]
\ee  
defines an oriented atlas of biLipschitz maps. Thus, we can write $[[ \,  M_j \times [0,h_j]  \, ]]$ as a countable sum of integrals as above using this atlas.

Let $\iota_j: M_j \times [0,h_j]   \to   Z$ and $\beta: M_j \times \{0\}   \to  M_j \times [ 0, h_j] $
be inclusion map which are $1$-Lipschitz, see  (\ref{region-dist-dec-to-Z}) for $\iota_j$.
Then   by the definition of $\varphi_0$
\begin{eqnarray}
\varphi_{0\#}[[ M_0]]    =    {\iota_j}_\sharp \beta_\sharp   [[M_j \times \{0\}]]   \in   \intcurr_m(Z).
\end{eqnarray}

\smallskip

Let   $\alpha: M_j \times \{h_j\}   \to  M_j \times [ 0, h_j]$  and $\tilde \alpha: M_j \setminus W_j  \to Z$ 
be  inclusion maps which are trivially Lipschitz,  $[[W_j  \times \{h_j \}  ]]  \in  \intcurr_m(M_j \times[0,h_j])$
 is the current obtained by restricting the atlas of $M_j \times \{h_j\}$ and in 
a similar way we get $[[M_j \setminus W_j]]   \in   \intcurr_m(M_j)$.
Then by the definition of  $\varphi_j$, 
\begin{eqnarray}
\varphi_{j\#}[[ M_j]]   =   {\iota_j}_\sharp    \alpha_\sharp [[W_j  \times \{h_j \}  ]]  +     \tilde \alpha_\sharp [[M_j \setminus W_j]].
\end{eqnarray}

\smallskip
Recall that  the inclusion map $\iota_j: M_j \times [0,h_j]   \to   Z$ is $1$-Lipschitz, hence the maps 
\be
\iota_j \circ    (\psi_i, \iota) :   A_{ik}\times [0,h_j]      \to   \iota_j  (  \psi_i(A_{ik}) \times [0,h_j]  )   \subset   \iota_j(M_j \times [0,h_j])
\ee
define an oriented atlas of Lipschitz maps  for $\iota_j(M_j \times [0,h_j])$,   where the maps can be considered to be biLipschitz as before. 
Then  we define $T$  as  the current with weight $1$ given by this oriented atlas
\be
T=  \sum  ( \iota_j \circ (\psi_i, \iota))_\sharp [[ 1_{A_{ik}\times [0,h_j]  } ]]= {\iota_j}_\sharp [[  \,M_j \times [0,h_j] \,]]
\ee
and
\be\label{Tprime}
T'=  \tilde \alpha_\sharp [[M_j \setminus W_j]]  -     {\iota_j}_\sharp    \alpha_\sharp [[   (M_j \setminus W_j  ) \times \{h_j \} ]] - [ [\, \partial M_j \times [0,h_j] \,]].
\ee
They are integral currents since their boundaries are the following currents  
\begin{align}
\partial T= & {\iota_j}_\sharp \partial[[ \, M_j \times [0,h_j] \,]]  \\
= &   {\iota_j}_\sharp    \alpha_\sharp [[ M_j  \times \{h_j \}  ]]  -   {\iota_j}_\sharp \beta_\sharp   [[M_j \times \{0\}]]  +   [ [\, \partial M_j \times [0,h_j] \,]], \,     
\end{align}
\begin{align}
\partial  T'  =  & \partial \left(        \tilde \alpha_\sharp [[M_j \setminus W_j]]  -  {\iota_j}_\sharp    \alpha_\sharp [[   (M_j \setminus W_j  ) \times \{h_j \} ]]        \right) -  \partial [ [\, \partial M_j \times [0,h_j] \,]]    \nonumber \\
 = & 0  -  \partial [ [\, \partial M_j \times [0,h_j] \,]]. 
\end{align}

Noticing that 
\begin{align}
T' = &  \varphi_{j\#}[[M_j]] -  {\iota_j}_\sharp    \alpha_\sharp [[ M_j \times \{h_j \} ]]  - [ [\, \partial M_j \times [0,h_j] \,]]  
\end{align}
and combining the equations above, we conclude that 
\be
\varphi_{j\#}[[M_j]]- \varphi_{0\#}[[ M_0]] =  T' + \partial  T. 
\ee
Thus, 
\be
d_{\mathcal{F}}(M_j, M_0) \le \mass(T') + \mass(T).
\ee

Now  since  $T=  {\iota_j}_\sharp [[ \,  M_j \times [0,h_j]   \,  ]]$ and $\iota_j:  M_j \times [0,h_j]  \to Z_j'$ is a 1-Lipschitz map, by 
(\ref{eq-pushMeasure}) we get
\begin{align}
\mass(T) \leq  & \mass([[\, M_j \times [0,h_j] \,]])  = \vol_j(M_j \times [0,h_j]) \leq h_jV. 
\end{align}
In a similar way,  from (\ref{Tprime}) and (\ref{eq-pushMeasure})  we get
\be
\mass(T')  \leq 2 \vol(M_j \setminus W_j) +     \vol( \partial M_j \times [0,h_j])  \leq 2 V_j +   Ah_j.
\ee
\end{proof}


\bibliographystyle{alpha}
\bibliography{allenBdry}

\end{document}